\documentclass[a4paper]{article}

\usepackage[utf8]{inputenc}  
\usepackage[T1]{fontenc}
\usepackage[english]{babel}
\usepackage[english=british]{csquotes}
\usepackage{amsmath,amsfonts,amssymb,amsthm,enumerate,tikz-cd,color}
\usetikzlibrary{matrix,arrows}
\usepackage{csquotes}
\usepackage{hyperref}      
\hypersetup{
	colorlinks=true, 
	bookmarksopen=true,            
	pdftitle={Quadratic Chabauty for modular curves}, 
	pdfauthor={Dogra, Le Fourn},     
}
\usepackage[all,cmtip]{xy} 

\theoremstyle{plain}
\newtheorem{Theorem}{Theorem}[section]
\newtheorem{lemma}[Theorem]{Lemma}
\newtheorem{Proposition}[Theorem]{Proposition}

\newtheorem{Corollary}[Theorem]{Corollary}

\theoremstyle{definition}
\newtheorem{Definition}[Theorem]{Definition}

\theoremstyle{remark}
\newtheorem{Remark}[Theorem]{Remark}
\newtheorem{Example}[Theorem]{Example}

\newcommand{\s}{\sigma}

\newcommand{\ag}{{\mathfrak{a}}}

\newcommand{\gn}{{\mathfrak{n}}}
\renewcommand{\ng}{{\mathfrak{n}}}

\newcommand{\gp}{{\mathfrak{p}}}

\newcommand{\gP}{{\mathfrak{P}}}

\newcommand{\Bcal}{{\mathcal B}}

\newcommand{\Dcal}{{\mathcal D}}

\newcommand{\Gcal}{{\mathcal G}}

\newcommand{\Lcal}{{\mathcal L}}
\newcommand{\Mcal}{{\mathcal M}}

\newcommand{\Ocal}{{\mathcal O}}
\newcommand{\Pcal}{{\mathcal P}}

\newcommand{\Scal}{{\mathcal S}}
\newcommand{\Tcal}{{\mathcal T}}

\newcommand{\N}{{\mathbb{N}}}
\newcommand{\Z}{{\mathbb{Z}}}
\newcommand{\Q}{{\mathbb{Q}}}

\newcommand{\C}{{\mathbb{C}}}

\newcommand{\F}{\mathbb{F}}

\newcommand{\Aut}{\operatorname{Aut}}

\newcommand{\Div}{\operatorname{Div}}
\newcommand{\End}{\operatorname{End}}
\newcommand{\Gal}{\operatorname{Gal}}
\newcommand{\GL}{\operatorname{GL}}
\newcommand{\SL}{\operatorname{SL}}
\newcommand{\Id}{\operatorname{Id}}

\newcommand{\ord}{\operatorname{ord}}

\newcommand{\Hom}{\operatorname{Hom}}

\newcommand{\gr}{\operatorname{gr}}

\renewcommand{\Re}{\operatorname{Re}}

\newcommand{\Jac}{\operatorname{Jac}}

\newcommand{\Pic}{\operatorname{Pic}}

\newcommand{\Tr}{\operatorname{Tr}}
\renewcommand{\mod}{\, \operatorname{mod} \,}

\newcommand{\Sel}{\operatorname{Sel}}
\newcommand{\Ker}{\operatorname{Ker}}

\newcommand{\loc}{\operatorname{loc}}
\newcommand{\Br}{\operatorname{Br}}
\newcommand{\Frob}{\operatorname{Frob}}
\newcommand{\Cl}{\operatorname{Cl}}
\newcommand{\lie}{\mathrm{Lie}}

\newcommand{\mt}{\mapsto}	
\newcommand{\lmt}{\longmapsto}
\newcommand{\ra}{\rightarrow}
\newcommand{\lra}{\longrightarrow}

\newcommand{\new}{\mathrm{new}}
\newcommand{\ab}{\mathrm{ab}}

\DeclareMathOperator{\rank}{rank}

\DeclareMathOperator{\sym}{Sym}

\newcommand{\dR}{\textrm{dR}}
\DeclareMathOperator{\red}{red}

\theoremstyle{remark}
\DeclareMathOperator{\ext}{Ext}

\DeclareMathOperator{\tr}{tr}
\DeclareMathOperator{\NS}{NS}

\newcommand{\et}{\mathrm{\acute{e}t}}
\DeclareMathOperator{\rk}{rk}
\DeclareMathOperator{\ns}{ns}
\newcommand{\CH}{\mathrm{CH}}

\DeclareMathOperator{\Coker}{Coker}

\newcommand{\AJ}{\mathrm{AJ}}

\DeclareFontFamily{U}{wncy}{}
\DeclareFontShape{U}{wncy}{m}{n}{<->wncyr10}{}
\DeclareSymbolFont{mcy}{U}{wncy}{m}{n}
\DeclareMathSymbol{\Sha}{\mathord}{mcy}{"58} 

\newcommand{\Qb}{\overline{\Q}}

\newcommand{\GalQ}{{\Gal(\Qb / \Q)}}
\newcommand{\spl}{\mathrm{s}}

\newcommand{\fonction}[5]{\begin{array}{c|ccl}           
		#1: & #2 & \longrightarrow & #3 \\
		& #4 & \longmapsto & #5 \end{array}}

\newcommand{\fonctionsansnom}[4]{\begin{array}{ccl}      
		#1 & \lra & #2 \\\
		#3 & \lmt & #4 
\end{array}}

\title{Quadratic Chabauty for modular curves and modular forms of rank one}
\author{Netan Dogra and Samuel Le Fourn\footnote{Supported  by the European Union’s Horizon 2020 research and programme under the Marie Sklodowska-Curie grant agreement No 793646, titled LowDegModCurve.}}

\begin{document}
	\maketitle
	
	\begin{abstract}
		In this paper, we provide refined sufficient conditions for the quadratic Chabauty method to produce a finite set of points, with the conditions on the rank of the Jacobian replaced by conditions on the rank of a quotient of the Jacobian plus an associated space of Chow--Heegner points.  We then apply this condition to prove the finiteness of this set for any modular curves $X_0^+(N)$ and $X_{\rm{ns}}^+(N)$ of genus at least 2 with $N$ prime.  The proof relies on the existence of a quotient of their Jacobians whose Mordell--Weil rank is equal to its dimension (and at least 2), which is proven via analytic estimates for orders of vanishing of L-functions of modular forms, thanks to a Kolyvagin-Logachev type result.
	\end{abstract}

	\tableofcontents
\section{Introduction}

The Chabauty--Kim method is a method for determining the set $X(\Q )$ of rational points of a curve $X$ over $\Q$ of genus bigger than 1. The idea is to locate $X(\Q )$ inside $X(\Q _p )$ by finding an obstruction to a $p$-adic point being global. The method developed in \cite{Kim05}, \cite{Kim09} produces a tower of obstructions
\[
X(\Q _p )\supset X(\Q _p )_1 \supset X(\Q _p )_2 \supset \ldots \supset X(\Q )
\]
In \cite{BDCKW}, it is conjectured that $X(\Q _p )_n =X(\Q )$  for all $n\gg 0$, and in \cite{Kim09} it is proved that standard conjectures in arithmetic geometry imply $X(\Q _p )_n $ is finite for all $n\gg 0$, but in general these results are not known.

The first obstruction set $X(\Q _p )_1$ is the one produced by Chabauty's method. In situations when $X(\Q _p )_1 $ is finite, it can often be used to determine $X(\Q )$.

The main results of this paper concern the finiteness of the Chabauty--Kim set  $X(\Q _p )_2 $ when $X$ is one of the modular curves $X_{\ns }^+ (N)$ or $X_0 ^+ (N)$ ($N$ a prime different from $p$), whose definition and properties we now recall briefly (more details are given in \S \ref{sectionfunctoriality}).

The curve $X_0 ^+ (N)$ is the quotient of $X_0 (N)$ by the Atkin--Lehner involution $w_N$. 
The curve $X_{\ns }^+ (N)$ is the quotient of $X(N)$ by the normalizer of a nonsplit Cartan subgroup. Determining the rational points of $X_{\ns }^+ (N)$ would resolve Serre's uniformity question \cite{Serre71}: is there an $N_0$ such that, for all $N>N_0$ and all elliptic curves $E$ defined over $\Q$ without complex multiplication, the mod $N$ Galois representation
\[
\rho _{E,N}:\Gal (\overline{\Q } / \Q )\to \Aut (E[N])
\]
is surjective? The Borel and normalizer of split Cartan subgroups of Serre's uniformity question have been given a positive answer respectively in \cite{Mazur77} and \cite{BiluParent11}.
 
As is explained in \S \ref{sectionfunctoriality}, in contrast to $X_0 (N)$ and $X_{\spl }^+ (N)$, for $X=X_0 ^+ (N)$ or $X=X_{\ns }^+ (N)$, it is expected that $X(\Q _p )_1$ is \textit{infinite}. 
 The main result of this paper is that we do obtain a finite set by refining the obstruction to `depth two'.
\begin{Theorem}
	\label{thm1}
	
	\hspace*{\fill}
	
	\begin{enumerate}
		\item For all prime $N$ such that $g(X_0 ^+ (N)) \geq 2$, $X_0 ^+ (N)(\Q _p )_2 $ is finite for any $p\neq N$.
		\item For all prime $N$ such that $g(X_{\ns }^+ (N)) \geq 2$ and $X_{\ns }^+ (N)(\Q )\neq \emptyset$, $X_{\ns }^+ (N)(\Q _p )_2 $ is finite for any $p\neq N$.
	\end{enumerate}
\end{Theorem}

\begin{Remark}
	\hspace*{\fill}
	\begin{itemize}
		\item 	For all primes $N$ for which one of the curves $X$ above has genus 0 or 1, $X(\Q)$ is infinite so this is the sharpest finiteness result for $X(\Q_p)_2$ one can expect.
		
		\item The only reason for the assumption that $X_{\ns }^+ (N)(\Q )$ is nonempty is that the definition of $X(\Q _p )_2 $ currently assumes that $X$ has a rational point (if Serre's uniformity question has a positive answer, then there are infinitely many $N$ for which $X_{\ns }^+ (N)(\Q )$ is empty). One can modify the definition of $X(\Q _p )_2 $ - for example in a similar manner to \cite{hain} - to remove this assumption, and then $X_{\ns }^+ (N)(\Q _p )_2 $ will be finite whenever the genus of $X_{\ns }^+ (N)$ is greater than 1. As this involves several techniques not relevant to the proof of Theorem \ref{thm1}, we do not pursue this point in this paper.
		
		\item Finally, results of \cite{BalakrishnanDogra2}, together with Edixhoven and Parent's explicit models for $X_{\ns }(N)$ \cite{edixhoven2019semistable}, allow to deduce from our result an explicit bound (polynomial in $N$) on the number of rational points on $X_0^+(N)$ and $X_{\ns }^+ (N)$, which we do in \S \ref{subseceffectivity}.
	\end{itemize}

\end{Remark}

As alluded to above, one can often use finiteness of $X(\Q _p )_n $ to determine $X(\Q )$ explicitly. 
The first motivation of this paper stems from the explicit determination of $X_{\rm{ns}}^+(13)(\Q)$ (starting from the finiteness of $X_{\rm{ns}}^+(13)(\Q_p)_2$) in \cite{Balakrishnanetc}. The finiteness of $X(\Q _p )_2 $ has also been used recently to determine the rational points of $X_0 ^+ (N)$ whenever it has genus two (in forthcoming work of Best--Bianchi--Triantafillou--Vonk) or genus three (in forthcoming work of Balakrishnan--Dogra--M\"uller--Tuitman--Vonk).

The proof of Theorem \ref{thm1} proceeds along the lines of the so-called \enquote{quadratic Chabauty method}, which requires a precise inequality (namely \eqref{eqQC}) in terms of invariants of the Jacobian $J$ of $X$ to hold (see \S \ref{subsecChowHeegner}). This inequality is expected to hold asymptotically for $X=X_0 ^+ (N)$ or $X=X_{\ns }^+ (N)$ conditionally on Birch and Swinnerton--Dyer conjecture (see \S  \ref{subsec_Samir_result}), but looks out of reach unconditionally for $N$ in noncomputable range. There are thus two important steps obtained in the proof of Theorem \ref{thm1}:
\begin{itemize}
	\item[$\bullet$] For $p$ a prime of good reduction of a smooth projective geometrically irreducible curve $X$ over $\Q$ with $X(\Q) \neq \emptyset$, $X(\Q_p)_2$ is finite under the condition that a similar inequality to \eqref{eqQC} holds not for $J$ but a quotient abelian variety $A$ of $J$, \textit{and} under an additional hypothesis $(C)$ on $X,J,A$.
	\item[$\bullet$] For $X=X_0^+(N)$ or $X=X_{\ns }^+ (N)$, there is an abelian variety of $A$ satisfying \eqref{eqQC} and such that $X,J,A$ satisfy $(C)$, if for $M=N$ (resp. $N^2$) there are two distinct normalised eigenforms $f \in S_2 (\Gamma _0 (M))^{+,\textrm{new}} $ such that $L'(f,1) \neq 0$.
\end{itemize}

The final input in the proof of Theorem \ref{thm1} is the following Theorem.

\begin{Theorem}\label{mainthm}
	For all $M=N$ or $N^2$ with $N$ prime, if the space $S_2 (\Gamma _0 (M))^{+,\rm{new}} $ is of dimension at least two, it contains two distinct normalised newforms $f$ such that $L'(f,1) \neq 0$.
\end{Theorem}

As explained in Remark \ref{remlemmatrickratio}, this result of nonvanishing is in fact quite weak compared to known or expected asymptotic estimates (giving a positive linear proportion of nonvanishing values) so the main difficulty in the proof of Theorem 2 lies in making such estimates effective enough to prove the result except for small enough $N$ so that the remaining cases can be checked algorithmically.

\subsection{Chow--Heegner points and quadratic Chabauty}
\label{subsecChowHeegner}

In general, $X(\Q _p )_n $ cannot unconditionally be proved to be finite without some assumptions on the Jacobian of $X$ (Kim showed that the Bloch--Kato conjectures imply that $X(\Q _p )_n $ is finite for all $n\gg 0$ \cite[Observation 2]{Kim09}). In the case $n=1$ (which reduces to the classical set-up of Chabauty's method) it is known that a sufficient condition is that
\begin{equation}
\label{eqC}
\rk(J) < \dim (J),
\end{equation} where $\rk(J)$ is the Mordell--Weil rank of $J(\Q)$. The simplest instance extending Chabauty's method when finiteness of $X(\Q _p )_n$ can be proved for $n >1$ is the following Lemma. To state the Lemma, define $J:= \Jac (X)$, and the Picard number $\rho(J)$ is the rank of the Néron--Severi group $\NS (J):=\Pic (J)/\Pic ^0 (J)$. By \cite[Proposition 17.2]{MilneAbVar86}, this is the same as the dimension of the subspace denoted by $\End^\dagger(J)$ of $\End ^0 (J):=\End(J) \otimes \Q$ consisting of endomorphisms that are symmetric, i.e. fixed by the Rosati involution.
\begin{lemma}[\cite{BalakrishnanDogra1}, Lemma 3.2] \label{lemma_NS}

	If 
	\begin{equation}
	\label{eqQC}
	\rk (J) < \dim(J) + \rho (J) - 1,
	\end{equation}
	then $X(\Q _p )_2 $ is finite. In particular, if $\rk (J) = \dim (J)$, then $X(\Q _p )_2 $ is finite whenever $\rho(J)>1$.
\end{lemma}

By Kolyvagin-Logachev type results due to Nekov\'a\v r and Tian (see Proposition \ref{propGZK} and its Corollary \ref{corKL}), Theorem \ref{mainthm} implies that the Jacobians of $X_0 ^+ (N)$ and $X_{\ns }^+ (N)$, which we will henceforth denote by $J_0 ^+ (N)$ and $J_{\ns }^+ (N)$ respectively, do have $\Q$-isogeny factors $A$ satisfying  $\rk (A) <\dim (A) + \rho (A)-1$, but it seems unattainable to prove unconditionally such a result for the full Jacobian. To deduce Theorem \ref{thm1}, we thus need a `quadratic Chabauty for quotients' result, analogous to the well-known fact that Chabauty's method also works under the relaxed condition $\rk (A)<\dim (A)$, i.e. \eqref{eqC} for an isogeny factor $A$ instead of $J$ (in fact, for modular curves, Mazur--Kamienny's method refines this for factors $A$ such that $\rk(A)=0$, see e.g. \cite{Baker99}).

As explained below, in general such a result seems non-trivial. Fix a basepoint $b\in X(\Q )$, and let $\AJ :X\to J$ be the corresponding Abel-Jacobi map. Let $A,B$ be abelian varieties over $\Q $, satisfying $\Hom _{\Q }(A,B)=0$, and suppose we have a surjection $(\pi _A ,\pi _B ):J \to A \times B$. 

A slight modification denoted by $\widetilde{\AJ}^*$ of the pullback by $\AJ$ (which basically amounts to considering the restriction of $\AJ$ on  symmetric line bundles, see \S \ref{subsecremindersNeronSeveri}) vanishes on $\Pic^0(J)$, so it factors through $\NS(J)$ and $\widetilde{\AJ}^*: \NS(J) \rightarrow \Pic(X)$ will denote this factorisation by abuse of notation. It induces a map 
\begin{equation}
\label{eqdefdpiA}
d_{\pi_A} : \NS(A) \overset{\widetilde{\AJ}^* \circ \pi_A^*}{\longrightarrow} \Pic(X) \overset{\deg}{\rightarrow} \Z
\end{equation}
 and therefore a map
\begin{equation}
\label{eqdeftheta}
\theta _{X,\pi _A ,\pi _B }:\Ker d_{\pi_A} \overset{\widetilde{\AJ}^* \circ \pi_A^*}{\longrightarrow} \Pic^0(X) \longrightarrow J(\Q) \overset{\pi_B \otimes \Q} \longrightarrow B(\Q )\otimes \Q,
\end{equation}
which is called the \textit{Chow-Heegner construction} (see Definition \ref{fundamental_definition} for details).

\begin{Remark}
	\label{remaltdef}
As an alternative definition (useful for the proofs), for any correspondence $Z \subset X\times X$, we can associate a cycle $D_Z (b)\in \Pic^0(X)$ (see \eqref{eqDZb}), and this defines a homomorphism $\NS(X \times X) \rightarrow \Pic^0(X)$ so that the composition
\[
\NS(J) \overset{(\AJ^{(2)})^*}{\longrightarrow} \NS(X \times X) \longrightarrow \Pic^0(X),
\]
where $\AJ^{(2)}:X \times X \rightarrow J$ is defined by $(x,y) \mapsto [x] + [y] - 2[b]$, is equal to $\widetilde{\AJ}^*$ on $(\widetilde{\AJ}^*)^{-1}(\Pic^0(X))$, which then allows to retrieve $\theta_{X,\pi_A,\pi_B}$ on cycles $Z$ coming from $\Ker d_{\pi_A}$.
\end{Remark}

 The `quadratic Chabauty for quotients' result that we prove in this paper says that we can replace $J$ with $A$, but the price we pay is that we replace $\rho(J)-1$ with the rank of $\Ker(\theta _{X,\pi _A ,\pi _B})$, which can be smaller than $\rho(A)-1$.
\begin{Proposition}\label{easy}
	Let $X$ be a curve as above. Suppose $J$ admits an isogeny $(\pi _A ,\pi _B):J\to A\times B$, where $\Hom (A,B)=0$. If 
	\begin{equation}\label{quotient_inequality} \tag{C}
	\rk (A) < \dim (A)+\rk (\Ker (\theta _{X, \pi _A ,\pi _B })),
	\end{equation}
	then $X(\Q _p )_2 $ is finite. 
\end{Proposition}

In the case where $\rk (A)=\dim (A)$ which we will focus on, we can simplify this condition in terms of \textit{nice correspondences}, defined in \S \ref{subsecremindersNeronSeveri}. More precisely, $(\pi_A, \pi_B)$ induces an isomorphism $\End^0(J) \cong \End^0(A) \times \End^0(B)$, and $X(\Q _p )_2 $ is finite whenever there exists a nontrivial nice correspondence $Z$ on $X\times X$ whose corresponding endomorphism of $J$ is zero in $\End^0(B)$, and whose corresponding Chow--Heegner point $D_Z (b) \in \Pic^0(X)$ is torsion when projected to $B$.

\begin{Remark}
	Note that, since $\rk (\Ker (\theta _{X, \pi _A ,\pi _B })) \leq \rho (A) -1$, inequality \eqref{quotient_inequality} implies that $A$ satisfies the naive analogue of Lemma \ref{lemma_NS}
	\begin{equation}\label{inequality_naive}
	\rk (A)< \dim (A)+\rho (A) -1.
	\end{equation}
	However, in general \eqref{quotient_inequality} is strictly stronger than \eqref{inequality_naive}. In fact, the trivial lower bound on $\rk(\Ker(\theta_{X,\pi_A,\pi_B})$ is $\rho(A)-1 - \rk(B)$ and if the latter was positive, it would imply \eqref{eqQC}. This is why Proposition \ref{propmodcurvessatisfyH} looks quite particular to modular curves. Moreover, understanding the rank of $\Ker (\theta _{X,\pi _A ,\pi _B })$ in general seems somewhat subtle - as becomes apparent in Example \ref{not_invariant} and \S \ref{subsecfuncC}, this quantity is not an invariant of the pair $(A,B)$, or even of the triple $(X,A,B)$, and does not seem to behave so well functorially even under quite strong hypotheses. 
Finally, as explained in the first appendix, this quantity is also related to the Gross--Kudla--Schoen cycles constructed in \cite{gross-schoen}.
	
	The following proposition emphasises that in fact, the supplementary condition $(C)$ can always be satisfied for our modular curves.
	
\end{Remark}

\begin{Proposition}
	\label{propmodcurvessatisfyH}
      Let $X=X_0 ^+ (N)$ or $X_{\ns }^+ (N)$, and $J=\Jac (X)$. Assume Theorem \ref{mainthm} holds, and the genus of $X$ is at least two. Then $J$ admit an isogeny $(\pi_A,\pi_B) : J \rightarrow A \times B$ satisfying
\begin{enumerate}
\item 
$\rk(A) = \dim A \geq 2$.
\item $\rho(A)>1$. 
\item $\rk(\Ker(\theta_{X,\pi_A,\pi_B})) =\rho (A)-1$.
\end{enumerate}
\end{Proposition}

As will become apparent in the proof, in fact we take $A$ to be the \textit{maximal} isogeny factor of $J$ whose analytic rank is equal to its dimension and $B$ its complement, otherwise we might not be able to ensure that the kernel of $\theta_{X,\pi_A,\pi_B}$ is nontrivial. This idea relies heavily on the use of (traces of) Heegner points on the modular curves $X_0(N),X_{\rm{ns}}(N)$, which generate $A(\Q)$ up to finite index, but will automatically be torsion in $B(\Q)$, both situations being ultimately by-products of the generalised Gross--Zagier formula (see section \ref{T2_to_T1}). 
Note that in this case the kernel of the theta morphism is not only nontrivial, but as large as it can be, which might indicate a deeper phenomenon at play.

The structure of the paper is as follows. In section \ref{sectionquadChabquot}, we give some reminders on N\'eron-Severi groups, Chow groups and correspondences, and describe the map $\theta_{X,\pi_A,\pi_B}$ in terms of cycles.  In section \ref{QCquotients} we prove Proposition \ref{easy}. In section \ref{sectionfunctoriality}, we prove Proposition \ref{propmodcurvessatisfyH} assuming Theorem \ref{mainthm}, after some discussion on \eqref{quotient_inequality}, and using generalised Gross--Zagier formulas. In section \ref{analytic_part}, we prove Theorem \ref{mainthm}. Finally, for sake of clarity and by lack of easily available references in the literature, we gather in Appendix \ref{AppendixChowHeegnerCeresa} results about the Chow-Heegner construction above and explain in Appendix \ref{AppendixKolyvaginLogachev} the proof of the Kolyvagin-Logachev type result needed to translate Theorem \ref{mainthm} into an algebraic rank result.

\subsection{Notation and conventions}

Unless stated otherwise, we adopt the following conventions in this paper.

$\bullet$ $X$ is a smooth projective geometrically irreducible curve of genus $\geq 2$ over $\Q$. $J$ is the Jacobian of $X$ and $\AJ: X \rightarrow J$ is the Albanese morphism with a fixed base point $b \in  X(\Q)$. The notation $\widetilde{\AJ}^*$ refers to twice the pullback on symmetric line bundles of $X$ to $\Pic(X)$ (see \eqref{eqdefAJmod}), and then factors through $\NS(J)$ (this is not the same as just the pullback $\AJ^*$ from $\Pic(J)$ to $\Pic(X)$, which does not vanish on $\Pic^0(J)$). 

$\bullet$ For any $n$ and any $S \subset \{1, \cdots, n\}$, the morphism 
\begin{equation}
\label{eqdefiSb}
i_S(b):X \rightarrow X^n
\end{equation}
is defined so that the $j$-th coordinate of $i_S(b)(x)$ is $x$ if $j \in S$ and $b$ otherwise. When there is no ambiguity on $b$ we denote it simply by $i_S$. Similarly, the morphism 
\begin{equation}
\label{eqdefpiS}
\pi_S: X^n \rightarrow X^{\# S}
\end{equation}
denotes the projection of $(x_1, \cdots, x_n)$ on the coordinates belonging to $S$.

$\bullet$ Morphisms between algebraic varieties over $\Q$ and their structures (line bundles, divisors, etc) are assumed to be defined over $\Q$.

$\bullet$ For a smooth projective algebraic variety $Y$ over $\Q$, $\NS(Y)$ is the N\'eron-Severi group of $Y$, and $\rho(Y):= \rk \NS(J)$ is the Picard number of $J$ (see \S \ref{subsecremindersNeronSeveri}).

$\bullet$ For any abelian variety $A$ over $\Q$ (in particular for $J$), 
$\rk(A)$ is the rank of the finite type $\Z$-module $A(\Q)$ and $\End^0 (A) := (\End_\Q A) \otimes \Q$.

$\bullet$ $N$ is a prime number (the level of our modular curves) and $M=N$ or $N^2$.

$\bullet$ $X_0(N)$ (resp. $X_{\rm{s}}^+(N)$, $X_{\rm{ns}}^+(N)$) is the modular curve quotient of $X_0(N)$ corresponding to the Borel structure (resp. quotient of $X(N)$ corresponding to the normaliser of split Cartan, normaliser of nonsplit Cartan), $X^+_0(N)$ is the quotient of $X_0(N)$ by the Atkin-Lehner $w_N$. Accordingly, their respective Jacobians are denoted by $J_0(N), J_{\rm{s}}^+(N), J_{\rm{ns}}^+(N), J_0^+(N)$ (see \S \ref{sectionfunctoriality}).

$\bullet$ For $X$ a variety over a field $K\subset \mathbb{C}$, $H^k (X,\mathbb{Z})$ refers to the singular cohomology of $X(\C)$.

$\bullet$ Given a unipotent group $U$, the central series filtration of $U$ is defined by $U^{(1)} = U$ and $U^{(i+1)}= [U,U^{(i)}]$, and $\gr _i (U):=U^{(i)}/U^{(i+1)}$ (in particular $\gr _1 (U)=U^{\ab }$). If a group $G$ acts continuously on $U$, then $G$ acts on the set of normal subgroups of $U$, and we say that a quotient $U/H$ is $G$-stable if the normal subgroup $H$ is stabilised by $G$. In this case there is a unique $G$-action on $U/H$ making the surjection $G$-equivariant.

$\bullet$ $p$ is a prime number different from $N$ which will be used (except in Appendix \ref{AppendixKolyvaginLogachev}) only in the context of $p$-adic numbers.

\subsection{Acknowledgements}

The authors wish to thank heartily Samir Siksek, who initiated this project and contributed to its progression, but declined to be listed as a co-author. He also graciously authorised us to include his original argument from his preprint \cite{Siksek17}, which is found in paragraph \ref{subsec_Samir_result}. We would also like to thank Daniel Kohen and Jan Vonk for helpful discussions.

\section{The quadratic Chabauty condition (C) for a quotient}
\label{sectionquadChabquot}

\subsection{Reminders on Chow groups and N\'eron--Severi groups}
\label{subsecremindersNeronSeveri}

We recall here the basic notions on correspondences of curves, and the Chow groups and Néron-Severi groups that we need. A good reference on correspondences is Smith's thesis \cite[Chapter 3]{Smiththesis}, and classical ones are \cite[section 11.5]{BirkenhakeLange} for the complex case and \cite[Chapter 16]{Fulton98} for the general case.

\begin{Definition}
	For any geometrically smooth and irreducible projective variety $Y$ over $\Q$ and any $k \leq \dim Y$: 
	
	\begin{itemize}
		\item  The \textit{Chow group} $\CH^{k}(Y)$ is the group of cycles of $Y$ of codimension $k$ up to rational equivalence.
		
		\item $c_k :\CH ^k (Y)\to H^{2k}(Y,\mathbb{Z})$ is the \textit{cycle map}, and 
$\CH^{k}_0(Y) :=\Ker (c_k )$ is its subgroup of homologically trivial cycles (in $Y(\C)$).
	\end{itemize}

	In particular, there are canonical isomorphisms 
	\[
	\CH^1(Y) \cong \Pic(Y), \quad \CH^1_0 (Y) \cong \Pic^0(Y).
	\] 
The N\'eron-Severi group $\NS(Y) := \Pic(Y)/ \Pic^0(Y)$ is thus embedded in $H^2(Y(\C),\Z)$.
		
\end{Definition}

We can also define a geometric \'etale cycle map \cite[Cycle]{SGA4half} 
\[
c_k ^{l,\et }:\CH ^k (Y) \to H^{2k}_{\et }(Y_{\overline{\Q }},\mathbb{Z}_l (k))
\]
and an absolute \'etale cycle map
\[
c_k ^{\mathrm{abs}}:\CH ^k (Y)\to H^{2k}_{\et }(Y,\mathbb{Z}_l (k)).
\]
By the Artin comparison theorem we have $\Ker (\prod _l c_k ^{l,\et })=\CH ^k _0 (Y)$. The \'etale Abel--Jacobi morphism is a homomorphism
\[
\AJ _{\et } :\CH ^k _0 (Y)\to \ext ^1 _{\GalQ}(\Q _p ,H^{2k-1}_{\et }(Y_{\overline{\Q }},\Q _p (k)))
\]
which may be defined using the Leray spectral sequence or (equivalently but more directly) by realising the extension class of a homologically trivial cycle $Z$ inside 
$H^{2k-1}((X-Z)_{\overline{\Q }},\Q _p (k))$ (see Jannsen \cite[II.9]{jannsen} or Nekovar \cite[5.1]{nekovar_height}).
By Poincar\'e duality, we may equivalently think of the target of $\AJ _{\et } $ as being
\[
\ext ^1 _{\GalQ}(H^{2(d-k)+1}_{\et }(Y_{\overline{\Q }},\Q _p (d)),\Q _p (k)) \quad (d = \dim Y).
\]
In particular, when $Y=X$ is a curve, and for $k=1$, the target of $\AJ_{\et}$ is 
\[
\ext^1_{\GalQ}(V_p(J),\Q_p(1)),
\]
where $J$ is the Jacobian of $X$ and $V_p(J) = T_p(J) \otimes_{\Z_p} \Q_p$.

Let us now review the basic definitions of correspondences.
\begin{Definition}
	\label{defcorr}
	For two curves $X_1,X_2$ as before:
	
	\begin{itemize}
		\item A \textit{correspondence} $Z$ on $X_1,X_2$ is a divisor of $\Div(X_1 \times X_2)$, \emph{prime} if the underlying divisor is. It is called \emph{fibral} if its prime components are horizontal or vertical divisors.
		
		\item If $Z$ is a nonfibral prime correspondence, the two projections $\pi_{1,Z}, \pi_{2,Z}: Z \rightarrow X_1, X_2$ are nonconstant so $\psi_Z:=(\pi_{2,Z})_* \circ \pi_{1,Z}^*$ defines a morphism from $\Div(X_1)$ to $\Div(X_2)$, inducing a morphism between the Jacobians of $X_1$ and $X_2$, and two rationally equivalent divisors define the same morphism. This defines by linearity  (extending to 0 for fibral prime divisors) a surjective morphism
		\begin{equation}
		\label{eqdefcorr}
        \operatorname{\psi} : \Pic (X_1 \times X_2) \rightarrow \Hom(\Jac(X_1), \Jac(X_2)),
		\end{equation}
		with kernel $\pi _1 ^* \Pic (X_1)\oplus \pi _2 ^* \Pic (X_2)$ with notation \eqref{eqdefpiS} (\cite[Theorem 11.5.1]{BirkenhakeLange} or \cite[Theorem 3.3.12]{Smiththesis}).
		\end{itemize}
\end{Definition}

When $X=X_1=X_2$, with the choice of a base point $b$, using notation from \eqref{eqdefiSb} and \eqref{eqdefpiS}, we obtain from $\pi_1 \circ i_1 = \Id_X$ and similar relations the identities
\begin{eqnarray}\label{Chow_decomp}
\Pic(X \times X) & = & \pi_1^*\Pic(X) \oplus \pi_2^* \Pic (X) \oplus \Ker (i_1^* \oplus i_2^*)\\
\Pic^0(X \times X) & = & \pi_1^*\Pic^0(X) \oplus \pi_2^* \Pic^0(X),
\end{eqnarray} 
 (see \cite[Proposition 3.3.8]{Smiththesis}) which induces a decomposition 
	\begin{equation}
	\label{kunneth}
	\NS(X \times X) = \pi_1^* \NS(X) \oplus \pi_2^* \NS(X) \oplus \Ker (i_1^* \oplus i_2^*),
	\end{equation}
	where the last direct factor then canonically identifies with $\End(J)$ via \eqref{eqdefcorr}. By abuse of notation, we thus denote 
	\[
	 \psi^{-1} : \End(J) \overset{\cong}{\rightarrow} \Ker (i_1^* \oplus i_2^*) 
	\]
    the inverse of this isomorphism. Now, the morphism $i_{1,2}^* - i_1^* - i_2^*$ is trivial when restricted to $\Pic^0(X \times X)$, hence induces a morphism 
	\begin{equation}
	\label{eqdiffdiag}
	\varphi: \NS(X \times X) \rightarrow \Pic(X).
	\end{equation}
	Define 
	\[
	 \fonction{\AJ^{(2)}}{X \times X}{J}{(x,y)}{[x]+[y]-2[b]}, \quad \widetilde{\AJ}^* := \varphi \circ (\AJ^{(2)})^*.	 
	\]
We have $\widetilde{\AJ }^* =[2]^* \circ \AJ ^* -2\AJ ^*$ so for $[\Lcal] \in \Pic(J)$,  
	\begin{equation}
	\label{eqdefAJmod}
	\widetilde{\AJ}^*([\Lcal]) = \AJ^*( [\Lcal]) + \AJ^* ([-1]^* [\Lcal]).
	\end{equation}
        using the classical identity $[n]^* (\mathcal{L})\simeq \mathcal{L}^{\otimes (\frac{n^2+n}{2})}\otimes [-1]^* (\mathcal{L}^{\otimes (\frac{n^2 -n}{2})})$.
	In particular, $\widetilde{\AJ}^*$ is twice the usual pullback by $\AJ$ on symmetric line bundles.
		
	For any divisor $D$ of $X \times X$, the degree of $\varphi(D)$ is equal to the rational trace of $\psi(D)$ (\cite[Proposition 11.5.2]{BirkenhakeLange}). This induces a morphism
	\[
	\widetilde{\theta}_{X,b} : \End(J)^{\rm{tr}=0} \overset{\varphi \circ \psi^{-1} }{\longrightarrow}  \Pic^0(X).
	\]
By \cite[IV.20]{MumfordAbVar}, the rule  $\mathcal{L}\mapsto \lambda _{\mathcal{L}}$ defined by $\lambda_\Lcal(P) = T_P^*\Lcal \otimes \Lcal^{-1} \in \Pic^0(J)$ induces an isomorphism
\begin{equation}\label{Rosat}
\fonction{\tilde{\lambda}}{\NS (J)}{\End^{\dagger}(J)}{[\Lcal]}{\Pcal^{-1} \circ \lambda_\Lcal}
\end{equation}
where $\Pcal:J \overset{\cong}{\rightarrow} \widehat{J}$ is a natural principal polarisation given by a theta divisor. This  the same as applying the composition $- \psi \circ (\AJ^{(2)})^*$. Indeed, via the natural morphisms $\widehat{J} \cong  \Pic^0(J)$ and $\Pic^0(X) \cong J$, the inverse $\widehat{J} \rightarrow J$ of the principal polarisation given by a theta divisor on $J$ is equal to $- \AJ^*$ from $\Pic^0(J)$ to $\Pic^0(X)$ (\cite[Proposition 11.3.5]{BirkenhakeLange}). 

Now, in terms of line bundles, by definition, given a line bundle $L$ on $X \times X$, the endomorphism of $\Pic(X)$ associated to it is given on points by $x \mapsto i_2^*(x)(L)$ with notation \eqref{eqdefiSb}. As $(\AJ^{(2)} \circ i_2(x)) = T_{[x]-[b]} \circ \AJ$, for a line bundle $\Lcal$ on $\Pic(J)$ and $x,y$ points of $X$ the endomorphism associated to $L=(\AJ^{(2)})^* \Lcal$ sends $[x]-[y]$ to 
\[
\AJ^* (T_{[x]-[b]}^* \Lcal - T_{[y]-[b]}^* \Lcal) = \AJ^* (T_{[x]-[y]}^* \Lcal - \Lcal) = \AJ^* \lambda_\Lcal([x]-[y]),
\]
which gives the equality up to $-1$. Hence, if we define 
\[
\theta _{X,b} : \NS(J)^0 :=\Ker (\deg \circ \widetilde{\AJ}^*) \overset{\widetilde{\AJ}^*}{\rightarrow}\NS (X)) \overset{ \widetilde{\AJ}^*}{\rightarrow} \Pic ^0 (X),
\]
we have the commutative diagram 
\[
 \xymatrix{\NS(J)^0 \ar^{- \theta_{X,b}}[rr] \ar[rd]_{\widetilde{\lambda}} & & \Pic^0(X) \\
 & \End(J)^{\rm{tr}=0}  \ar[ru]_{\widetilde{\theta}_{X,b}}&
 }
\]

\begin{Remark}
In \cite{Balakrishnanetc}, an element of $\Pic (X\times X)$ whose image under \eqref{kunneth} lies in $\End ^\dagger (J)^{\tr =0}$ is referred to as a `nice correspondence'. 
\end{Remark}

\subsection{Chow--Heegner points and diagonal cycles}
\label{subsecChowHeegnerdiagonalcycles}

We recall an equivalent version of the morphism $\widetilde{\theta }_{X,b}$, which appears in \cite{DR} and \cite{BalakrishnanDogra1}. As our discussion applies in fairly broad generality, we take $X$ to be a smooth geometrically irreducible projective curve over a field $K$ of characteristic zero.
Fix $b\in X(K)$, and $S \subset \{ 1,\ldots n\}$, let $X_S$ denote the image of $X$ under the closed immersion $i_S (b)$ defined in \eqref{eqdefiSb}. 
For any $Z \in \Div(X \times X)$, let $C_Z(b) := (i_{\{1,2\}}^*(b) -i_{\{1 \}}^*(b) -i_{\{2 \} }^*(b) )(Z)$  and
\begin{equation}
\label{eqDZb}
D_Z (b) := C_Z(b)- \deg (C_Z(b)) \cdot b \in \Pic^0(X).
\end{equation}
We refer to $D_Z (b)$ and $C_Z (b)$ as \textit{Chow--Heegner} points, following \cite{DRS}.

The map $Z\mapsto D_Z (b)$ factors through $\Pic (X\times X)$, and has the following relation to $\widetilde{\theta }_{X,b}$.
The projection
\[
\Pi :\Pic (X\times X)\to \Ker (i_1 ^* \oplus i_2 ^* )
\]
associated to \eqref{Chow_decomp} is given by $(1-\pi _1 ^* \circ i_1 ^* -\pi _2 ^* \circ i_2 ^* )$, giving the identity
\[
i_{\{ 1,2 \} }^* \circ \Pi =i_{\{1,2\} }^* - i_1 ^* -i_2 ^* .
\]
Since $\deg (C_Z (b))=\deg (\varphi (\Pi ([Z])))$, for any $Z$ in $\Pic (X\times X)$ which lies in the kernel of $\deg \varphi$, we have
\begin{equation}\label{DZ_theta}
D_Z (b)=C_Z (b)=\theta _{X,b}(\psi \circ \Pi ([Z])).
\end{equation}
We define $Z^t \in \CH ^1 (X\times X)$ to be the pull-back of $Z$ under the involution
\begin{align*}
X\times X & \to X\times X \\
(x,y) & \mapsto (y,x). \\
\end{align*}
\begin{lemma}\label{basepoint_dependence}
In the notation of Definition \ref{defcorr}, we have
\[
D_Z (b')-D_Z (b)=\psi _Z (b-b')+\psi _{Z^t }(b-b').
\]
\end{lemma}
\begin{proof}
We have $i_{\{1,2\}}(b)=i_{\{1,2\}}(b')$. Hence
\[
C_Z (b')-C_Z (b)=i_{\{1\}}(b)^*(Z) -i_{\{ 1\}}(b')^*(Z) + i_{\{2\}}(b)^*(Z) -i_{\{ 2\}}(b')^*(Z).
\]
By definition of the correspondences, we then have 
\[
(i_{\{1\}}(b)^* -i_{\{1 \} }(b')^* )(Z)=\psi_Z (b-b'),
\]
and 
\[
(i_{\{2\}}(b)^* -i_{\{2 \} }(b')^* )(Z)=\psi_{Z^t} (b-b'),
\]
which proves the equality for $C_Z(b') - C_Z(b)$, thus for $D_Z(b') - D_Z(b)$ as the degrees are then equal.
\end{proof}
\begin{Definition}\label{fundamental_definition}
Given a surjective homomorphism $\pi_B: J\to B$ of abelian varieties, we obtain a homomorphism 
\begin{equation}
\label{eqintermtheta}
\Ker (\NS (J) \overset{\deg \circ \widetilde{\AJ}^*}{\longrightarrow} \Z) \overset{ \widetilde{\AJ}^*}{\longrightarrow} \Pic^0(X) \longrightarrow J \overset{\pi_B}{\longrightarrow} B.
\end{equation}
By Lemma \ref{basepoint_dependence} and \eqref{DZ_theta}, for a divisor $Z$ on $X \times X$, if $\psi_{\Pi(Z)}$ has image contained in $\Ker (\pi _B )$, then the image of $[Z]$ in $B$ via \eqref{eqintermtheta} is independent of the choice of basepoint. In particular, if we have a surjection $(\pi _A ,\pi _B):J\to A\times B$, and $\Hom (A,B)=0$, then we obtain a homomorphism independent of $b$, which we will denote by
\begin{align*}
\theta _{X,\pi _A,\pi _B}: & \Ker (d_{\pi _A }) \to B \\
& [L] \mapsto \pi _B \circ \widetilde{\theta }_{X,b}\circ \pi _A ^* ([L]).
\end{align*}
\end{Definition}
\begin{Remark}
This construction also has a direct description in terms of line bundles, although this is not the one we use to calculate $\theta _{X,\pi _A ,\pi _B}$ in examples.
Given a line bundle $L$ on $A$ whose pull-back to $X$ via $\AJ ^* \circ \pi _A ^* $ has degree zero, we may also consider the projection of $\AJ ^* \circ \pi _A ^* (L)$ to $B$. Variants of this construction are studied in the thesis of Michael Daub \cite{daub}. By \eqref{eqdefAJmod}, we have the identity \cite[Proposition 3.3.3]{daub} 
\[
\theta _{X,\pi _A ,\pi _B }=[2] \circ \pi _B \circ \AJ ^* \circ \pi _A ^*,
\]
in particular the right-hand side does vanish on $\Pic^0(A)$ \cite[Proposition 3.3.2]{daub}.
\end{Remark}

\begin{Example}\label{not_invariant}
Note that $\theta _{X,\pi _A ,\pi _B }$ is not an invariant of $A $ and $B$, or even of $X,A,B$. For example, let $A$ and $B$ be distinct isogeny factors of $X_0 (N)$, and let $X=X_0 (N^2 )$. Let $f_1 ,f_2 :X\to X_0 (N)$ be the two natural morphisms, and let $(\pi _{A_i },\pi_{B_i })$ be the morphisms $\Jac (X)\to A\times B$ obtained by composing the surjection $J_0 (N)\to A\times B$ with $f_{i*}$. Then $\theta _{X,\pi _{A,i} ,\pi _{B,i} }$ can be nonzero (see \cite{DR} for examples), however if $i\neq j$, $\theta _{X,\pi _{A,i},\pi _{B,j}}$ is identically zero, since for any choice of line bundle $[L]$ in $\NS (A)$, the associated point $D_{[L]}(b)$ will lie in $f_{i}^* J_0 (N)$, hence the projection to $f_{j*}J_0 (N)$ will be torsion.
\end{Example}

\section{Proof of finiteness of the Chabauty--Kim set under (C)}\label{QCquotients}
The strategy of proof of Proposition \ref{easy} is very similar to that of \cite[Lemma 3.2]{BalakrishnanDogra1}.
To explain this strategy, we need to establish some notation. $X,A,B$ are as in the proposition. 
Define
\[
V:=T_p (J)\otimes \Q _p , \quad V_A := T_p (A) \otimes \Q _p , \quad V_B :=T_p (B) \otimes \Q _p.
\]
Let $U_n (b)$ denote the maximal n-unipotent quotient of the $\Q _p $-unipotent fundamental group of $\overline{X}$ at some basepoint $b$ as defined in \cite[\S 10]{Deligne89}. Let $U$ be a Galois-stable quotient of $U_n (b)$ (i.e. a quotient by a Galois-stable normal subgroup of $U_n(b)$). Let $T_0 $ be the set of primes of bad reduction for $X$, and let $T=T_0 \cup \{ p \}$. Denote the maximal quotient of $\Gal (\overline{\Q } /\Q )$ unramified outside $T$ by $G_{\Q ,T}$, and for $v\in T$ denote $\Gal (\overline{\Q }_v /\Q _v )$ by $G_{\Q _v} $. Then by \cite{Kim05},\cite{Kim09}, we have a commutative diagram
\[
\begin{tikzpicture}
\matrix (m) [matrix of math nodes, row sep=3em,
column sep=3em, text height=1.5ex, text depth=0.25ex]
{X(\Q ) & H^1 (G_{\Q ,T},U) \\
\prod _{v\in T}X(\Q _v ) & \prod _v H^1 (G_{\Q _v },U). \\};
\path[->]
(m-1-1) edge[auto] node[auto]{} (m-2-1)
edge[auto] node[auto] { $j$ } (m-1-2)
(m-2-1) edge[auto] node[auto] {$\prod _{v\in T} j_v $ } (m-2-2)
(m-1-2) edge[auto] node[auto] {$\prod _{v\in T}\loc _v $} (m-2-2);
\end{tikzpicture}
\]
with the following properties
\begin{enumerate}
\item For $G=G_{\Q ,T}$ or $G_{\Q _v }$, and all $i<k$, the sets $H^1 (G,U^{(i)}/U^{(k)})$ have the structure of $\Q _p $ points of an algebraic variety, so that the algebraic structure on $H^1 (G,\gr _i U)$ is just the usual scheme structure on a vector space, and the maps
\[
H^1 (G,\gr _i U)\to H^1 (G,U/U^{(i+1)})\to H^1 (G,U/U^{(i)})
\]
come from morphisms of algebraic varieties. The maps $\loc _v $ are then algebraic for these structures.

\item For $v\in T_0 $, the map $j_v $ has finite image.
\item The image of the map $j_p $ is contained inside the subvariety $H^1 _f (G_{\Q _p },U)$ of crystalline torsors.

\end{enumerate}
The following Lemma is proved in \cite[Lemma 3.1]{BalakrishnanDogra1} (although the result is stated only in the case $A=J$, the proof generalises to the case where $A$ is an arbitrary quotient of $J$).
\begin{lemma}\label{dimension_inequality}
Let $U$ be a Galois-stable quotient of $U_2 (b)$. Suppose $U$ is an extension of $V_A$ by $\Q _p (1)^n$, where $A$ is some abelian variety over $\Q$ and $V_A= T_p (A) \otimes \Q _p$. If
\[
\rk (A(\Q ))<n+\dim (A),
\]
then $X(\Q _p )_2$ is finite. In particular, if $\rk (A(\Q ))=\dim (A)$, then $X(\Q _p )_2 $ is finite whenenever $n> 0$.
\end{lemma}
To prove Proposition \ref{easy}, we construct a quotient $U$ of $U_2(b)$  
as in Lemma \ref{dimension_inequality}, with $n=\rk (\Ker \theta_{X,\pi_A,\pi_B})$.
We again take $X$ to be a smooth projective geometrically irreducible curve over a field $K$ of characteristic zero.

The group $U_2 (b)$ is an extension 
\begin{equation}\label{U2_extn}
1 \to \Ker (H^2 (J_{\overline{\Q }},\Q _p )\stackrel{\AJ ^* }{\longrightarrow }H^2 (X_{\overline{\Q }},\Q _p ))^* \to U_2 (b) \to V \to 1.
\end{equation}
Hence for any $\xi \in \Ker (\NS (J) \overset{\widetilde{\AJ}^*}{\to} \NS (X))$, we may quotient by the kernel of the dual of the Chern class $c_p ^{\et }(\xi) \in H^2(X_{\Qb},\Q_p(1))$ (see \S \ref{subsecChowHeegner})
\[
c_p ^{\et }(\xi )^* (1):\Ker (H^2 (J_{\overline{\Q }},\Q _p )\stackrel{\AJ ^* }{\longrightarrow }H^2 (X_{\overline{\Q }},\Q _p ))^* \to \Q _p (1)
\]
to obtain a quotient $U_Z$ of $U_2 (b)$ which is an extension of $V$ by $\Q _p (1)$. Similarly, for any nice correspondence on $X\times X$, we obtain a quotient of $U_2 (b)$ which is an extension of $V$ by $\Q _p (1)$.

\begin{lemma}[\cite{BalakrishnanDogra1}, Theorem 6.3]\label{DZB_HM}
Let $U$ be a Galois-stable quotient of $U_2 (b)$  of the form
\[
1\to \Q _p (1)\to U\to V_p(J) \to 1,
\]
coming from a correspondence $Z\subset X\times X$ as above. Then the associated extension class of $\lie (U)$ in $\ext ^1 _{G_K }(V_p(J),\Q _p (1))$ is equal to the \'etale Abel--Jacobi class of the cycle $D_Z (b)$ (see \S \ref{subsecremindersNeronSeveri}).
\end{lemma}
\begin{proof}
Let $\mathcal{E}(\lie (U))$ be the universal enveloping algebra of $\lie (U)$, and let $I(\lie (U))$ be the kernel of the co-unit morphism $\mathcal{E}(\lie (U))\to \Q _p $.
In \cite[\S 6]{BalakrishnanDogra1}, a Galois representation $E_Z$ is constructed as a quotient of $\mathcal{E}(\lie (U))$. The image of $I(\lie (U))$ in $E_Z$ is an extension $IE_Z $ of $V$ by $\Q_p (1)$. By \cite[Theorem 6.3]{BalakrishnanDogra1}, the extension class of $IE_Z$ in $\ext^1_{\GalQ}(V_p(J),\Q_p(1))$ is the  Abel--Jacobi class of $D_Z (b)$. The restriction of $I(\lie (U))\to IE_Z$ to $\lie (U)\subset I(\lie (U))$ is an isomorphism, and hence the extension class of $\lie (U)$ is isomorphic to $D_Z (b)$.
\end{proof}
As explained in Appendix \ref{AppendixChowHeegnerCeresa}, Lemma \ref{DZB_HM} is really a consequence of Hain and Matsumoto's computation of the extension class of $\lie (U_2 )$ in terms of the Ceresa cycle.
Hence to complete the proof of Proposition \ref{easy}, it will be enough to prove the following Lemma.
\begin{lemma}
Let $U'$ denote the quotient of $U_2 $ obtained from the surjection $\gr _2 (U_2 )\to \Ker (d_{\pi_A} )^* \otimes \Q _p (1)$.
There exists a Galois stable quotient $U$ of $U'$ which is an extension of $V_A$ by $\Ker (\theta _{X,\pi _A ,\pi _B })$:
\[
\begin{tikzpicture}
\matrix (m) [matrix of math nodes, row sep=3em,
column sep=3em, text height=1.5ex, text depth=0.25ex]
{1 & \Ker (d _{\pi_A} )^* \otimes \Q _p (1)  & U' & V_A \oplus V_B & 1 \\
1 & \Ker (\theta _{X,\pi _A ,\pi _B})^* \otimes \Q _p (1) & U & V_A & 1.
 \\};
\path[->]
(m-1-1) edge[auto] node[auto]{} (m-1-2)
(m-1-2) edge[auto] node[auto]{} (m-2-2)
edge[auto] node[auto] {  } (m-1-3)
(m-1-3) edge[auto] node[auto]{} (m-2-3)
edge[auto] node[auto] {  } (m-1-4)
(m-1-4) edge[auto] node[auto]{} (m-2-4)
edge[auto] node[auto] {  } (m-1-5)
(m-2-1) edge[auto] node[auto]{} (m-2-2)
(m-2-2) edge[auto] node[auto]{} (m-2-3)
(m-2-3) edge[auto] node[auto]{} (m-2-4)
(m-2-4) edge[auto] node[auto]{} (m-2-5);
\end{tikzpicture}
\]
\end{lemma}
\begin{proof}
It will be enough to prove the corresponding statement for the Lie algebra $L'$ of $U'$.
The commutator map 
\[
[\cdot,\cdot]_{U'}:(V_A \oplus V_B )\times (V_A \oplus V_B )\to \Ker (d_{\pi_A} )^* \otimes \Q _p (1)
\]
is the composite of the commutator on $U_2 $, given by 
\[
(V_A \oplus V_B )\times (V_A \oplus V_B ) \to \Coker (\Q _p (1) \stackrel{\cup ^* }{\longrightarrow } \wedge ^2 V_A \oplus V_A \otimes V_B \oplus \wedge ^2 V_B )
\]
with the surjection
\[
\Coker (\Q _p (1) \stackrel{\cup ^* }{\longrightarrow } \wedge ^2 V_A \oplus V_A \otimes V_B \oplus \wedge ^2 V_B ) \to \Ker (d_{\pi_A} )^* \otimes \Q _p (1)
\]
Since the latter map factors through projection onto $\wedge ^2 V_A /\Q _p (1)$, the composite map factors through projection onto $V_A \times V_A $. 
Hence for any quotient $Q$ of $\Ker (d_{\pi _A })^* \otimes \Q _p (1)$, we can construct a Lie algebra quotient of $L'$ which is an extension of $V_A$ by $Q$. It remains to show that, when $Q=\Ker (\theta _{X,\pi _A ,\pi _B})$, we can make this quotient Galois stable. That is, we first quotient out by 
$(\Ker (d_{\pi _A })/\Ker (\theta _{X,\pi _A ,\pi _B }))^* \otimes \Q _p (1)$, to form an extension
\[
0\to \Ker (\theta _{X,\pi _A ,\pi _B })^* \otimes \Q _p (1)\to L''\to V_A \oplus V_B \to 0.
\]
The surjection $L''\to V_B $ induces a Galois equivariant short exact sequence of Lie algebras
\[
0\to L'\to L''\to V_B \to 0,
\]
and to construct the quotient $U\to U'$, it is enough to show that this short exact sequence admits a Galois equivariant section. Here $L'$ sits in a short exact sequence
\[
0\to \Ker (\theta _{X,\pi _A ,\pi _B})^* \otimes \Q _p (1)\to L' \to V_A \to 0,
\]
and since $L''/\Ker (\theta _{X,\pi _A ,\pi _B})^* \otimes \Q _p (1)=V_A \oplus V_B$, it is enough to show that image of $[L'']$ under the composite map
\[
\ext ^1 _{G_{\Q }}(V_A \oplus V_B ,\Ker (\theta _{X,\pi _A ,\pi _B})^* \otimes \Q _p (1))\to \ext ^1 _{G_{\Q }}(V_B ,\Ker (\theta _{X,\pi _A ,\pi _B})^* \otimes \Q _p (1))
\]
is zero.

Equivalently, we want to show that $\Ker (\theta _{X,\pi _A ,\pi _B})$ is contained in the kernel of the homomorphism 
\[
\Ker (d_{\pi _A })\to \ext ^1 _{G_{\Q }}(V_B ,\Q _p (1))
\]
sending $\xi \in \Ker (d_{\pi _A })$ to the $V_B$ component of the extension class in $\ext ^1 (V_A \oplus V_B ,\Q _p (1))$ associated to the quotient of $L'$ defined by $c _p ^{\et }(\xi )$:
\[
\begin{tikzpicture}
\matrix (m) [matrix of math nodes, row sep=3em,
column sep=3em, text height=1.5ex, text depth=0.25ex]
{0 & \Ker (d_{\pi _A })^* \otimes \Q _p (1)  & L' & V_A \oplus V_B & 0 \\
0 & \Q _p (1) & c_{p} ^{\et }(\xi )^{*} (L) & V_A \oplus V_B & 0. \\};
\path[->]
(m-1-1) edge[auto] node[auto]{} (m-1-2)
(m-1-2) edge[auto] node[auto]{$c_p ^{\et }(\xi )^*\otimes \Q _p (1)$} (m-2-2)
edge[auto] node[auto] {  } (m-1-3)
(m-1-3) edge[auto] node[auto]{} (m-2-3)
edge[auto] node[auto] {  } (m-1-4)
(m-2-1) edge[auto] node[auto]{} (m-2-2)
(m-2-2) edge[auto] node[auto]{} (m-2-3)
(m-2-3) edge[auto] node[auto]{} (m-2-4)
(m-2-4) edge[auto] node[auto]{} (m-2-5)
(m-1-4) edge[auto] node[auto]{} (m-2-4)
edge[auto] node[auto] {  } (m-1-5);
\end{tikzpicture}
\]
By Lemma \ref{DZB_HM}, this extension class is equal to the \'etale Abel--Jacobi class of $D_{c_p ^{\et }(\xi )}(b)$, and hence its $V_B$ component is equal to the \'etale Abel--Jacobi class of $\theta _{X,\pi _A ,\pi _B }(c_p ^{\et }(\xi ))$. Under the hypothesis, the latter is 0 so the extension class is trivial, which concludes the proof of Proposition \ref{easy}.
\end{proof}

\subsection{Bounding the number of rational points on curves satisfying (C)}
\label{subseceffectivity}
Following \cite{BalakrishnanDogra2}, the proof of finiteness of $X(\Q _p )_2$ may be used to prove an explicit upper bound on $\# X(\Q _p )_2$. To explain this, we introduce some notation. By \cite[Corollary 1]{KT08}, for all $v\neq p$, the size of the image of $X(\Q _v )$ in $H^1 (G_{\Q _v },U_2 )$ is finite, and is equal to one for all primes of good reduction for $X$. Let $T_0$ denote the set of primes of bad reduction for $X$, and for $v\in T_0$ let $n_v $ denote the size of the image of $X(\Q _v )$ in $H^1 (G_{\Q _v },U_2 )$.

\begin{Corollary}\label{effective_cor}
Suppose $X$ satisfies the hypotheses of Proposition \ref{easy}, and furthermore that the rank of $A(\Q )$ is equal to its dimension, and the $p$-adic closure of $A$ has finite index in $A(\Q _p )$. Let $n:=\prod _{v\in T_0 }n_v $.
Let $D$ be an effective divisor on $X$, let $Y\subset X_{\mathbb{Z}_p }$ be the complement of the support of a normal crossings divisor on $Y$ with generic fibre $D$, and let $\{ \omega _0 ,\ldots ,\omega _{2g-1}\}$ be a set of differentials in $H^0 (X,\Omega (D))$ forming a  basis of $H^1 _{\dR} (X)$. Then there are $a_{ij},a_i \in \Q _p$, $\eta \in H^0 (X,\Omega (D))$ and $g\in H^0 (X,\Omega (2D))$, and $\alpha _1 ,\ldots ,\alpha _n $ in $\Q _p $, such that
\begin{equation}\label{explicit_equation}
X(\Q _p )_2 \cap Y(\mathbb{Z}_p )\subset \bigcup _{i=1}^n \{x\in Y(\mathbb{Z}_p ): \sum a_{ij}\int ^x _b \omega _i \omega _j +\sum   a_i \int ^x _b \omega _i +\int^x _b \eta +g(x)=\alpha _i \}.
\end{equation}
\end{Corollary}
\begin{proof}
The argument is identical to the proof of \cite[Proposition 6.4]{BalakrishnanDograII}, however as the hypotheses are different we explain the steps.
Arguing as in loc. cit, there are $b_{ij}$, $b_i $ in $\Q _p $ such that $X(\Q _p )_2 \cap Y(\mathbb{Z}_p )$ is contained in the finite set of $x\in Y(\mathbb{Z}_p )$ satisfying
\[
h_p (A_Z (x))-\sum b_{ij}\left(\int  ^x _b  \omega _i \right) \left(\int  ^x _b  \omega _j \right)-\sum \int ^x _b \omega _i =-\sum _{v\in T_0 }h(A_Z (b)^{\phi _v }) ,
\]
for some $(\phi _v )$ in $\prod _{v \in T_0 }j_v (X(\Q _v )).$ Here $A_Z (b)^{(\phi _v )}$ denotes the twist of $A_Z (b)$ by $\phi _v $.

Hence we deduce \eqref{explicit_equation} from the formula for $h_p (A_Z (x))$ given in \cite[Lemma 6.7]{BalakrishnanDograII}, and the formula
\[
\left(\int  ^x _b  \omega _i \right) \left(\int  ^x _b  \omega _j \right)=\int  ^x _b  \omega _i \omega _j +\int ^x _b \omega _j \omega _i .
\]
\end{proof}

\begin{Corollary}
Suppose $X$ satisfies the hypotheses of Proposition \ref{easy}, and furthermore that the rank of $A(\Q )$ is equal to its dimension. Then
\[
\# X(\Q ) <\kappa_p \left(\prod _{v\in T_0 }n_v \right)\# X(\mathbb{F}_p )(16g^3+15g^2-16g+10),
\]
where $\kappa _p :=1+\frac{p-1}{p-2}\frac{1}{\log (p)}$.
\end{Corollary}
\begin{proof}
It is enough to prove that, for all $x_0 \in X(\mathbb{Z}_p )$, we can choose $D,\omega _i $ such that $\overline{x}:=\red (x_0 )$ lies in $Y(\mathbb{F}_p )$, and 
\begin{align*}
& \# \{x\in \red ^{-1}(\{\overline{x}\})\subset X(\Q _p ):  \sum a_{ij}\int ^x _b \omega _i \omega _j +\sum ^x _b  a_i \int  ^x _b \omega _i +\int ^x _b  \eta +g(x)=0 \} \\
 < & \kappa_p (16g^3+15g^2-16g+10).
\end{align*}
This follows from \cite[Proposition 3.2]{BalakrishnanDogra2} together with \cite[\S 4, below Lemma 4.4.]{BalakrishnanDogra2}.
\end{proof}
\begin{Remark}
In \cite{Betts19}, it is proved that the size of $j_{2,v}(X(\Q _v ))$ can be bounded by the number of irreducible components of a regular semistable model of $X$ over a finite extension of $\Q _v $. Hence using work of Edixhoven and Parent on stable models of $X_{\ns }^+(N)$ \cite{edixhoven2019semistable}, one can use the above corollary, together with Theorem 1, to give explicit bounds on the size of $X_{\ns }^+(N)$ and $X_0 ^+ (N)$.
\end{Remark}
\subsection{Functoriality properties of $(C)$}
\label{subsecfuncC}
The heart of the proof of Proposition \ref{mainproposition} is an interpretation of diagonal cycles on $X_0 (N)$ and $X_{\ns }(N)$ in terms of Heegner points. The following Lemma allows us to use this to deduce something about diagonal cycles on $X_0 ^+ (N)$ and $X_{\ns }(N).$
This lemma is a special case of a theorem of Daub \cite[Proposition 3.3.5]{daub}.

\begin{lemma}\label{functoriality_of_diagonal_cycles}
\hspace*{\fill}

\begin{enumerate}
\item
Let $f:X'\to X$ be a non-constant morphism of curves over a field $K$. Suppose $b'\in X'(K)$ maps to $b\in X(K)$ under $f$, and let $Z$ be an element of $\CH^1 (X\times X)$. Then
\[
D_{(f,f)^* Z}(b')=f^* (D_Z (b)).
\]
\item Let $f:X'\to X$ be a non-constant morphism of curves over a field $K$, and let $f_*$ denote the induced surjection $J':=\Jac (X')\to J:=\Jac (X)$. Let $(\pi _A ,\pi _B)$ be a surjective homomorphism from $J$ to $A\times B$. Then 
\[
\rk (\Ker (\theta _{X,\pi _A ,\pi _B}))=\rk (\Ker (\theta _{X' ,\pi _A \circ f_* ,\pi _B \circ f_*})).
\]
\end{enumerate}
\end{lemma}
\begin{proof}
For $*=\{1\},\{2\}$ or $\{1,2\}$, the diagram
\[
\begin{tikzpicture}
\matrix (m) [matrix of math nodes, row sep=3em,
column sep=3em, text height=1.5ex, text depth=0.25ex]
{ X' & X'\times X' \\
 X & X\times X \\ };
\path[->]
(m-1-1) edge[auto] node[auto] {$i_* (b')$} (m-1-2)
edge[auto] node[auto] {$f$} (m-2-1)
(m-2-1) edge[auto] node[auto] {$i_* (b)$} (m-2-2)
(m-1-2) edge[auto] node[auto] {$(f,f)$} (m-2-2);
\end{tikzpicture}
\]
commutes.
Hence we obtain, in $\CH ^1 (X')$,
\begin{align*}
f^* (C_Z (b)) & =(f^* \circ i_{\{1,2\}}(b)^* -f^* \circ i_{\{1 \}}(b)^*-f^* \circ i_{\{2\}}(b)^* )(Z)\\
& = (i_{\{1,2\}}(b')^* \circ (f,f)^*  - i_{\{1 \}}(b')^*  \circ (f,f)^* - i_{\{2\}}(b)^*  \circ (f,f)^*  )(Z) \\
& = C_{(f,f)^* (Z)}(b')
\end{align*}
and the result follows for $D_Z(b)$. The second item follows from the first, by \eqref{DZ_theta}.
\end{proof}

Note that while the behaviour of diagonal cycles under pull-backs is tautological, their behaviour under push-forwards is not. For this reason it seems difficult to deduce statements about diagonal cycles on $X_{\ns }(N)$ from results on $X_{\spl }(N)$, in spite of the explicit isogeny relating their Jacobians explained below.

\section{\texorpdfstring{Proof of (C) for $X_0 ^+ (N)$ and $X_{\ns }^+ (N)$}{Proof of (C) for X0(N)+ and Xns(N)+}}
\label{sectionfunctoriality}
Given Proposition \ref{easy}, it will be enough to prove Theorem \ref{mainthm}, and the following.

\begin{Proposition}\label{mainproposition}
Assume Theorem \ref{mainthm}. Then, for $X=X_0 ^+ (N)$ or $X_{\ns }^+ (N)$ of genus at least 2, there exists an isogeny
\[
(\pi _A ,\pi _B ):J\to A\times B,
\]
where $\rk (A) = \dim (A) = \rho(A) \geq 2$ and such that, for all $L$ in $\Ker (d_{\pi _A })$, 
$
\theta _{X,\pi _A ,\pi _B }(L)=0
$
is torsion (see Definition \ref{defiHeegnerquotient} for the choices of $A$ and $B$).
\end{Proposition}

We recall the definitions of some of the modular curves which appear, for example, in \cite{Chen}. Define $C_{\ns }^+ (N),C_{\spl }^+ (N)$ to be normalisers in $\GL _2 (\mathbb{Z}/\N \mathbb{Z})$ of fixed choices of non-split Cartan $C_{\rm{ns}}(N)$ and split Cartan subgroups $C_{\rm{s}}(N)$ of $\GL _2 (\mathbb{Z}/N\mathbb{Z})$. The (normaliser of) split and nonsplit Cartan modular curves are defined by 
\[
X_{\ns }^+ (N) :=X(N)/C_{\ns }^+ (N), \quad X_{\spl }^+ (N)=X(N)
/C_{\spl }^+ (N).
\]
Similarly we define $X_{\ns }(N)$ and $X_{\spl }(N)$ to be the quotients of $X(N)$ by $C_{\ns }(N)$ and $C_{\spl }(N)$ respectively.
Since $C_{\ns }(N)$ and $C_{\spl }(N)$ contain the centre of $\GL _2 (\mathbb{Z}/N\mathbb{Z})$ and their determinant goes through all $(\Z/N\Z)^*$, all $X_{\ns }(N)$, $X_{\spl }(N)$ and their Atkin--Lehner quotients are geometrically connected and defined over $\Q$.

Non-cuspidal $K$-points of $X_{\spl }(N)$ (for $K$ a field of characteristic zero) correspond to elliptic curves $E$ together with a pair $C_1 ,C_2 $ of cyclic subgroups of $E$ of order $N$ generating $E[N]$.
We have an isomorphism 
\begin{equation}\label{split_atkin}
X_0 (N^2 )\simeq X_{\spl }(N),
\end{equation}
which sends a point $(f:E\to E' )$ to $(E'',C_1 ,C_2 )$, where $E'' :=E/(N\cdot \Ker (f))$, $C_1 $ is the image of $\Ker (f)$ in $E''$, and $C_2 $ is the image of $E[N]$ in $E''$. 

The curve $X_{\spl }(N)$ is naturally a degree two cover of $X_{\spl }^+ (N)$, and there is an isomorphism $X_{\spl }^+ (N)\simeq X_0 ^+ (N^2 )$ compatible with \eqref{split_atkin}.

\subsection{Jacobians of modular curves and the asymptotics of the quadratic Chabauty condition}
\label{subsec_Samir_result}

We recall a formula for the Picard numbers and ranks of modular Jacobians and their quotients, due to Siksek \cite{Siksek17}.
Let $\mathcal{B}_{N^k} $ denote a normalised eigenbasis for the space of newforms in $S_2 (\Gamma _0 (N^k ))$. Let $\mathcal{B}_{N^k }/\GalQ$ denote a choice of representatives of the orbits of $\Bcal_{N^k}$ under $\GalQ$. We denote by $\mathcal{B}_{N^k }^+$ the subset of $\mathcal{B}_{N^k }$ with Atkin--Lehner eigenvalue $1$ for $w_{N^k }$. The Jacobians $J_0 (N^k )^{\new }$ and $J_0 ^+ (N^k )^{\new }$ admit isogenies
\[
J_0 (N^k )^{\new }  \sim \prod _{f\in \mathcal{B}  _{N^k} /\GalQ} A_f , \quad 
J_0 ^+ (N^k )^{\new } \sim \prod _{f\in \mathcal{B}^+ _{N^k} /\GalQ} A_f ,
\]
where $A_f $ denotes the $\Q $-simple abelian variety associated to $f$ by the Eichler--Shimura correspondence (which is independent of the choice of representative of the orbit). 
Because $X_s^+(N)$ is isomorphic to $X_0^+(N^2)$ as we have seen above,
\begin{align*}
J_s^+(N) \cong J_0 ^+ (N^2 ) \sim J_0 (N)\times J_0 ^+ (N^2 )^{\new }
\end{align*}
and by a theorem of Chen \cite[Theorem 1]{Chen}, we also have a $\Q$-isogeny
\begin{equation}
\label{eqChenisogeny}
J_{\ns }^+ (N)\sim J_0 ^+ (N^2 )^{\new }.
\end{equation}

The following lemma says that one would not expect to be able to use Chabauty's method to understand $X(\Q )$.
\begin{lemma}
Let $X=X_0 ^+ (N)$ or $X_{\ns }(N)$. Then the weak Birch--Swinnerton-Dyer conjecture implies $X(\Q _p )_1 =X(\Q _p )$.
\end{lemma}
\begin{proof}
The weak Birch--Swinnerton-Dyer conjecture implies that, for $f\in \mathcal{B}_{N^k}$, $A_f$ will have positive rank whenever $f$ has positive analytic rank. Since $f\in \mathcal{B}_{N^k}$ has odd analytic rank whenever $w_{N^k}(f)=1$, and $A_f$ is simple over $\Q $, the Birch--Swinnerton-Dyer conjecture hence implies that every isogeny factor of $\Jac (X)$ (over $\Q $) has positive rank.

Since $\End (A_f )$ is an order in the totally real field $K_f$, every isogeny factor of $\Jac (X)$ has rank at least equal to its dimension. To prove the lemma, we must show that the image of $A_f (\Q )$ in $\lie (A_f )_{\Q _p }$ under the $p$-adic logarithm map generates $\lie (A_f )_{\Q _p }$ as a $\Q _p $-vector space. This is equivalent to the statement that the image of $A_f (\Q )$ in $\lie (A_f )_{\mathbb{C}_p }$ generates the latter as a $\mathbb{C}_p $-vector space. Since $\lie (A_f )_{\overline{\Q }}$ decomposes as a sum of one-dimensional isotypic components $\lie (A_f )_{\overline{\Q },g}$, for $g$ conjugate to $f$, and the $p$-adic logarithm is $\End (A_f )$-equivariant, we deduce that if the image of $A_f (\Q )$ does not span $\lie (A_f )_{\mathbb{C}_p }$ then there is a $g$ conjugate to $f$ such that the image of $A_f (\Q )$ in $\lie (A_f )_{\mathbb{C}_p ,g}$ is zero. By the $p$-adic analytic subgroup theorem \cite[Theorem 1]{matev2010p}, \cite[Theorem 2.2]{fuchs2015p} if $P\in A_f (\overline{\Q } )$ has the property that $\log (P)\in \lie (A_f )_{\mathbb{C}_p }$ lies in a proper subspace defined over $\overline{\Q }$, then $P$ lies in a proper commutative sub-variety $B\subset A_{f,\overline{\Q }}$. Hence we deduce that if $A_f (\Q )$ does not generate $\lie (A_f )_{\mathbb{Q}_p }$, then $A_f (\Q )$ lies in a proper commutative subvariety of $A_{f,\overline{\Q }}$, since the isotypic components of $\lie (A_f )_{\mathbb{C}_p }$ are defined over $\overline{\Q }$.

We claim that this contradicts the Birch--Swinnerton-Dyer conjecture. More generally, if $A$ is a simple abelian variety over $\Q $ and $\pi :A_K \to B$ is a non-zero morphism of abelian varieties over a finite Galois extension $K|\Q $, we claim that $P\in A(\Q )$ is torsion if and only if its image in $B(K)$ is torsion (in particular, when $A=A_f $ and $B$ is an isogeny factor, we deduce that $A_f$ has rank zero over $\Q $ if and only if there is as isogeny factor $B$ of $A_{f,\overline{\Q }}$ such that the image of $A_f (\Q)$ in $B$ is torsion). To see this claim, for $\sigma \in \Gal (K|\Q )$ let $\pi ^{\sigma }$ denote the conjugate homomorphism $A_K \to B^{\sigma }$. If $\pi (P)$ is torsion then $\pi ^{\sigma }(P)=\pi (P)^{\sigma }$ is torsion for all $\sigma $, hence the image of $P$ under the map
\[
\prod _{\sigma \in \Gal (K|\Q )}\pi ^{\sigma }:A_K \to \prod _{\sigma }B^{\sigma }
\]
is torsion. However, this map descends to a non-zero morphism of $\Q $, and hence by simplicity of $A$, if $\pi (P)$ is torsion then $P$ is torsion.
\end{proof}

Moreover, two abelian varieties $A_f$, $A_g$ for $f,g \in \Bcal_{N^k}$ are non-isogenous unless $f$ and $g$ are conjugate by $\GalQ$, and $\End^{\dagger }(A_f)$ is always totally real of rank $\dim (A_f )$, which proves that each of the Jacobians $J = J_0^+(N),J_{\spl}^+(N), J_{\rm{ns}}^+(N)$ satisfies $\rho(J) = \dim J$, and hence the condition \eqref{eqQC} becomes 
	\begin{equation}\label{PropSamir}
	\rk(J) < 2 \cdot \dim(J) - 1
	\end{equation}
	(for a more general such condition for modular curves, see the main result of \cite{Siksek17}).
Using the isogenies above, the Birch--Swinnerton-Dyer conjecture implies
\[
\rk(J_0 ^+ (N)) = \sum_{f \in \Bcal_N^+} \ord_{s=1} L(f,s), \quad \rk(J_{\rm{ns}}^+(N)) =  \sum_{f \in \Bcal_{N^2}^+} \ord_{s=1} L(f,s).
\]

There is a whole literature on analytic estimates for these types of analytic ranks. In particular, using \cite[Theorem 1.4]{KowalskiMichelVanderKam} one can show that the Birch--Swinnerton-Dyer conjecture implies that
\[
\limsup_{N} \frac{\rk(J_0^+(N))}{\dim J_0^+(N)} \leq 1.3782,
\]
and in particular asymptotically that \eqref{eqQC} is always satisfied. It is likely that the same result can be obtained for $J_{\rm{ns}}^+(N)$, but the square level (we are looking at $J_0^+(N^2)^{\rm{new}}$) raises serious technical difficulties for analytic estimates of second moments used there.

On the other hand, by Corollary \ref{corKL}, Theorem \ref{mainthm} implies that we have an isogeny factor $A$ of $J$ satisfying $\rho (A)>1$ and $\rk (A)=\dim (A)$, hence to prove Proposition \ref{mainproposition} it suffices to construct a nonzero $[L] \in \Ker (\NS (A)\to\NS (X))$ satisfying $\theta _{X,\pi _A ,\pi _B}([L])=0$, where $B$ is the isogeny factor consisting of modular abelian varieties associated to modular forms whose analytic rank of $L$-functions is greater than 1. It will be shown that for any $L$, its image $\theta _{X,\pi _A ,\pi _B}(L)$ can be represented by a divisor supported on cusps and Heegner points, and hence is torsion by the generalised  Gross--Zagier formula (\cite[Theorem 6.1]{zhang2})
This motivates the following definition.

\begin{Definition}[Heegner quotient]
	\label{defiHeegnerquotient}
	\hspace*{\fill}
	
	Let $M=N$ or $N^2$. The \textit{Heegner quotient} $A$ of $J_0(M)^{\rm{new}}$ is the product 
	\[
	A := \prod_{\substack{f \in \Bcal^{+,\rm{new}}_{M}/\GalQ\\ L'(f,1) \neq 0}} A_f, 
	\] 
	and its \textit{complement} is 
	\[
	B := \prod_{\substack{f \in \Bcal^{+,\rm{new}}_M/\GalQ \\ L'(f,1) = 0}} A_f 
	\]
	(so that$A \times B$ is isogenous to $J_0^+(M)^{\rm{new}}$, not the full $J_0(M)^{\rm{new}}$).
	
\end{Definition}

	In particular, Corollary \ref{corKL} implies that $\rk(A) = \dim(A)$ (assuming the Birch and Swinnerton-Dyer conjecture, it is the largest factor of $J_0^+(M)$ with this property) and the generalised Gross--Zagier formula implies that all images of traces of Heegner points on $X_0 (N)$ in $B$ are torsion (see \S \ref{T2_to_T1} and \S \ref{T2_to_T1_ns}). In the case of $X_{\ns }(N)$, there is also a notion of Heegner point due to Kohen and Pacetti, inspired by the points used in Zhang's Gross--Zagier formula for $X_{\ns }(N)$ (and more general Shimura curves). 
		
The main result of the next section is the following lemma, which refers to $X_0 (N)$ and $X_{\ns }(N)$ rather than their Atkin--Lehner quotients. However, by Lemma \ref{functoriality_of_diagonal_cycles} it implies Proposition \ref{mainproposition}.
\begin{lemma}\label{when_is}
Let $X=X_0 (N)$ or $X_{\ns }(N)$, and $A,B$ the Heegner quotient and its complement as defined above, endowed with the natural projections $(\pi _A ,\pi _B ):\Jac (X)\to A\times B.$ Then for all $[L]$ in $\Ker (d_{\pi _A })$,
$
\theta _{X,\pi _A ,\pi _B }([L])
$
is torsion. In particular the rank of the kernel of $\theta_{X,\pi_A,\pi_B}$ is maximal (in particular at least 1 if $\dim A \geq 2$).
\end{lemma}

\subsection{How to prove (C) using Heegner points under the analytic hypothesis: $X=X_0 (N)$}\label{T2_to_T1}
In this section we prove Lemma \ref{when_is}. We will deduce it from the Gross--Zagier--Zhang theorem. In the case of $X_0 (N)$, as explained in \cite{daub} or \cite{DRS}, we could also deduce it from the Yuan--Zhang--Zhang formula for the height of diagonal cycles (see \S \ref{alternative}). By a \textit{Heegner point} on $X_0 (N)$ we will mean a point 
\[
E\to E'
\]
 on $Y_0 (N)$ such that $E$ and $E'$ have CM by the same order of an imaginary quadratic field $K$, not necessarily maximal but assumed to be with conductor prime to $N$ (see \cite{Gross84} for a review of their properties, in particular $N$ has to be split or ramified in $K$). 
 
 An eigenform $f \in S_2 (\Gamma _0 (N))^{+,\textrm{new}}$ defines by Eichler-Shimura theory a $\Q$-simple quotient $\pi : J_0(N) \rightarrow A_f$ of $J_0(N)$ (in fact of $J_0 ^+ (N)$) and the Heegner points behave on $A_f$ in the following way.
 
 \begin{lemma}
 	\label{lemmaHeegnerpoints}
 	\hspace*{\fill}
 	
 	 \begin{enumerate}
 		\item If $L'(f,1)\neq 0$, then $\rk(A_f)= \dim (A_f )$ (and $A_f(\Q)$ is generated by the projection of a trace of a suitable choice of Heegner point).
 		\item If $L'(f,1)=0$, then for any $P$ in $\Div ^0 (X_0(N))(\overline{\Q } )^{\Gal (\overline{\Q }|\Q )}$ supported on the set of Heegner points, the image $\pi(P)$ is torsion in $A_f (\Q )$.  
 	\end{enumerate}
 \end{lemma}
 
 \begin{Remark}
 	The original Gross--Zagier formula \cite[Theorem I.6.3]{GrossZagier86} is not sufficient for the second part of the Lemma, as it only deals with Heegner points for which the discriminant of the order is squarefree (in particular, the order is maximal) and prime to $N$, which we cannot afford to assume here. This is why we need Zhang's formula and the ensuing technical interpretation.
 \end{Remark}
 \begin{proof}
 
 The first part is given by Proposition \ref{propGZK}. The second part is a consequence of the generalised Gross--Zagier formula of Zhang \cite[Theorem 6.1]{zhang2} which for this case is made completely explicit in \cite[Theorem 1.1]{CaiShuTian14}, see also \cite[Example after Theorem 1.5]{CaiShuTian14}. We use the following notation: $f \in S_2(\Gamma_0(N))$ is a normalised eigenform, $K$ an imaginary quadratic field number field in which $N$ is not inert,  $c$ prime to $N$, $\Ocal_c = \Z + c \Ocal_K$, and $1_c$ the trivial ring class character on $\Pic(\Ocal_c)$. We denote by $H_c$ the ray class field of $K$ with conductor $c$. If $P$ is a Heegner point on $X_0(N)$ with CM by $\Ocal_c$, it belongs  to $X_0(N)(H_c)$, and we define 
 \[
 P_{1_c} = \sum_{\s \in \Gal(H_c/K)} (P_{1_c}^\sigma -[\infty]) \in J_0(N)(K) \subset J_0(N)(H_c).
 \]
  On the other hand, if $ J(H_c) \otimes \C$ denotes the extension of scalars of $J(H_c )$
 endowed with the extended Néron-Tate height, we have the decomposition into isotypical components
 \[
 J_0(N)(H_c) \otimes \C = \bigoplus_{g} J_0(N)_{g},
 \]
 where $g$ goes through all eigenforms of weight 2 of $J_0(N)$, so that $J_0(N)_g$ is exactly the isotypical part where $T_n$ acts by multiplication by $a_n(g)$. We denote by $P_{1_c}^f$ the projection of $P_{1_c}$ on the $f$-isotypical component. The statement of \cite[Theorem 1.1]{CaiShuTian14} then tells (which is sufficient for us) that $L'(f,1_c,1)$ as defined there is proportional to the extended Néron-Tate height of $P_{1_c}^f$.

We have the equality of L-functions
 \[
 L(f,1_K,s) = L\left(f,s\right) L\left(f \otimes \chi_K,s \right),
 \] 
 with $1_K$ the trivial class character on $\Pic(\Ocal_K)$ and $\chi_K$ the Dirichlet character associated to $K$. In particular (and given the signs of functional equations on the right), our hypothesis $L'(f,1)=0$ guarantees that $L(f,1_K,s)$ vanishes with order at least 2 at $1$, so the left-hand side of \cite[Theorem 1.1]{CaiShuTian14} is zero for $c=1$. This also holds for any $c$ prime to $N$, because by construction $L(f,1_{c},s)$ is a multiple of $L(f,1_K,s)$ around $1$ (given the definition again). We have thus proved that $P_{1_c}^g$ is zero in $J_0(N)(H_c) \otimes \C$.
 
Now, the group $\operatorname{Aut}(\C)$ acts on $J_0(N)(H_c) \otimes \C$ by the identity on the left and the natural action on the right, and for every $\alpha \in \Aut(\C)$ acting as such, we have $P_{1_c}^\alpha = P_{1_c}$ and then  for every $\alpha \in \Aut(\C)$, we obtain $(P_{1_c}^g)^{\alpha} = P_{1_c}^{\alpha(g)}$ where $\alpha(g)$ is the eigenform obtained by conjugating the coefficients of $g$ (see \cite[Corollary V.1.2]{GrossZagier86}). Now, as we also have the decomposition 
 \[
 J(H_c) \otimes \C \cong \prod_{g / \GalQ} A_g(H_c) \otimes \C
 \] 
 in subrepresentations of the Hecke algebra, the sum of all $P_{1_c}^g$ for $g$ conjugate to $f$ is proportional to the projection $\pi$ of the trace of $P- (\infty)$ (belonging to $J_0(N)(K)$) in $A_f(K) \otimes \C$, so we have proven that this projection in $A_f(K)$ is torsion.
 
\end{proof}
We now explain how to deduce Lemma \ref{when_is} from this result.
Let $m$ be an integer coprime to $N$. Define the Hecke correspondence $\widetilde{C}_{m}$ to be the image of $X_0 (mN)$ in $X_0 (N)\times X_0 (N)$ under the product of the two natural maps $X_0 (mN) \to X_0 (N)$. We define 
\[
C_{m}=(1-\pi _1 ^* i_1 ^* -\pi _2 ^* i_2 ^* )\widetilde{C}_{m}
\]
to be the projection of $\widetilde{C}_{m}$ onto the $\End (J_0 (N))$ component of $\Pic (X_0 (N)\times X_0 (N))$ (see \eqref{eqdefcorr}). Then $C_{m}$ lands in the subspace $\NS (J_0(N)) \subset \End (J_0 (N))$ of endomorphisms symmetric with respect to the Rosati involution. When $m$ is square-free, $C_{m}$ is the Hecke operator $T_{m}$. In general, $C_{m}$ is a linear combination of $T_{m/d}$ for $d$ divisors of $m$.

Recall that $i_{1,2} :X_0 (N)\hookrightarrow X_0 (N)\times X_0 (N)$ denotes the diagonal morphism. A non-cuspidal point in the support of $i_{1,2} ^* (\widetilde{C}_{m} )$ is a cyclic $N$-isogeny $f:E_1 \to E_2 $, together with cyclic subgroups $G_i$ of $E_i $ of order $m$ such that $f(G_1 )=G_2 $, and isomorphisms
\[
E_i \stackrel{\simeq }{\longrightarrow }E_i /G_i 
\]
which commute with $f$ and the induced isogeny $E_1 /G_1 \to E_2 /G_2 $. In particular, the ring of endomorphisms of each $E_i$, of discriminant denoted by $D_i$, thus contains an element of norm $m$ so there exist $A_i ,B_i $ in $\mathbb{Z}$ for which
\begin{equation}
\label{eqDi}
A_i ^2 + D_i B_i ^2 =4m.
\end{equation}
The isogeny being cyclic, $A_i$ and $B_i$ must be coprime here. The point $E_1 \to E_2 $ is a Heegner point of $Y_0(N)$ if and only if $D_1 =D_2 $.

\begin{lemma}\label{diagonalHeegner}
Let $X=X_0 (N)$, let $m$ be prime to $N$, and let $\widetilde{C}_m$ be the Hecke correspondence defined above. Then the divisor $i_{1,2} ^* \widetilde{C}_{m}$ is supported on the set of Heegner points whenever $m$ is less than $N/4$. 
\end{lemma}
\begin{proof}
Let $(E_1 \to E_2 )$ be a non-cuspidal point in the support of $i_{1,2}^* \widetilde{C}_m$ as above. Suppose the point is not Heegner.
Since $E_1 $ and $E_2 $ are $N$-isogenous, $D_2 = \lambda^2 D_1$ for some rational $\lambda>0$ a power of $N$. Since $\lambda \neq 1$, we must  have $D_i$ divisible by $N^2$ for some $i$, and hence $m>N^2 /4$, by \eqref{eqDi}. Finally, if the conductor of the order was not prime to $N$, we would also have $N^2 | D_i$ which leads to the same inequality.
\end{proof}

By the following Lemma (essentially just the Sturm bound) we have enough Hecke operators $C_m$ for which $i_{1,2}^* C_m$ is supported on cusps and Heegner points to complete the proof of the first part of Lemma \ref{when_is}.
\begin{lemma}\label{span_endomorphisms}
Let $N$ be a prime. Then, any element of $\End ^\dagger (J_0 ^+ (N))^{\tr =0}$, viewed as a subspace of $\End ^\dagger (J_0 (N))^{\tr =0}$, can be written as a $\Z$-linear combination of endomorphisms associated to the Hecke correspondences $C_m$, for $m<N^2 /4$ prime to $N$.
\end{lemma}
\begin{proof}
By the Sturm bound (\cite{Stein07} Theorem 9.18), the set of Hecke operators $T_m$ for $m<N^2 /4$ spans the Hecke algebra of endomorphisms of $J_0(N)$. Since $a_N (f)=-1$ on newforms such that $f_{|w_N} = -f$, the set of Hecke operators $T_m$ for $m<N^2 /4$ prime to $N$ spans the Hecke algebra of endomorphisms of $J_0 ^+ (N)$ (which is the full endomorphism algebra over $\Q$).
\end{proof}

This completes the proof of case (1) of Proposition \ref{mainproposition}. Indeed, Lemma \ref{span_endomorphisms}
implies that any nice correspondence $Z$ on $X_0 (N)$ can be written as a linear combination of the $C_m $ for $m<N^2 /4$ prime to $N$. By Lemma \ref{diagonalHeegner}, for any such $Z$, $D_Z (b)$ is supported on Heegner points and cusps, so by Lemma \ref{lemmaHeegnerpoints} (part 2), its image by $\pi_B$ is torsion.

\subsection{How to prove (C) using Heegner points under the analytic hypothesis: $X=X_{\ns }^+ (N)$}\label{T2_to_T1_ns}
The second case is similar to the first, but we must replace the classical notion of Heegner point with Heegner points on non-split Cartan modular curves in the sense of Zhang/Kohen--Pacetti, and replace Gross--Zagier--Zhang on $X_0 (N)$ with Zhang's Gross--Zagier theorem on $X_{\ns }(N)$.

To make results easier to state, we use the moduli interpretation of $X_{\rm{ns}}(N)$ and $X_{\rm{ns}}^+(N)$ given in \cite{KohenPacetti} and its consequences. To do so, one fixes an $\varepsilon \in \F_N$ which is not a square. A pair $(E,\phi_\varepsilon)$ is then an elliptic curve $E$ together with an endomorphism $\phi_\varepsilon$ of $E[N]$ whose square is multiplication by $\varepsilon$. Such an endomorphism has eigenvalues in $\F_{N^2} \backslash \F_N$, and two pairs $(E,\phi_\varepsilon)$ and $(E',\phi_\varepsilon')$ are isomorphic if there is an isomorphism $\psi : E \rightarrow E'$ such that on $E[N]$, $\psi \circ \phi_\varepsilon = \phi_{\varepsilon}' \circ \psi$.

$X_{\rm{ns}}(N)$ is the compactified moduli space of such pairs up to isomorphism \cite[\S 1.2]{KohenPacetti}. Furthermore, the natural involution on this modular curve is given by $(E,\phi_\varepsilon) \mt (E, - \phi_\varepsilon)$.

First, we define Hecke correspondences $\widetilde{C_m} \subset X_{\ns }(N)\times X_{\ns }(N)$ (for $m$ prime to $N$) as follows. We have a curve $X_{\ns }(N,m)=X_{\ns }(N)\times _{X(1)}X_0 (m)$ given by adding an auxiliary $\Gamma_0(m)$ structure. We have two maps $X_{\ns }(N,m)\to X_{\ns }(N)$, the forgetful one, and the one sending $(E, \phi_\varepsilon ,C)$ to $(E/C,\overline{\pi_C} \circ \phi_\varepsilon \circ \overline{\pi_C}^{-1})$ where $C$ is a cyclic subgroup of order $m$, $\pi_C : E \rightarrow E/C$ the natural projection, and $\overline{\pi_C}$ the induced map $E[N] \rightarrow (E/C)[N]$.  Furthermore, Chen morphisms between $J_{\rm{ns}}(N)$ and $J_0(N^2)$ are equivariant with respect to the Hecke actions \cite[Theorem 1.11]{KohenPacetti}.

We will again use the generalised Gross--Zagier formula from Zhang from \cite{zhang2}, in a slightly different context here. We follow the notation of \cite[\S 6]{zhang2}.  Let $K/\Q $ be an imaginary quadratic field inert at $N$ (instead of split or ramified in the previous case), and let $K\hookrightarrow M_2(\Q)$ be an embedding associated to an integral basis of $\Ocal_K$. For a choice of order $\Ocal_c$ of $K$ of conductor $c$ prime to $N$, define
\[
R_c=\mathcal{O}_c +N\cdot  M_2(\Z)
\]
(notice the index of $N \Ocal_K$ is $N^2$). The Shimura variety $M_{U_c}$ is then uniformised as
\[
M_{U_c} (\mathbb{C})=\GL_2(\Q )_+ \backslash \mathcal{H} \times \GL_2(\mathbb{A}_f )/U_c,
\]
where $U_c$ can be defined as $\GL_2(\Z_v)$ for places $v$ not dividing $N$, and $(R_c \otimes \Z_N)^* \subset \GL_2(\Z_N)$ at $N$ (seen in $\GL_2(\Z_N)$). Note that $\GL_2(\Q)^+ \cdot U_c =\GL_2(\mathbb{A}_f )$ and $\GL_2(\Q )_+ \cap U_c \subset \SL_2(\Z)$ contains the subgroup $\Gamma(N)$ of $\SL_2(\Z)$ of all matrices congruent to the identity modulo $N$, and the quotient is a conjugate of $C_{\ns }(N) \cap \SL_2(\Z/N\Z)$, where the precise choice of $C_{\rm{ns}}(N)$ comes from the reduction modulo $N$ of $\Ocal_c$ inside $M_2(\Z/N\Z)$ given by the embedding (it is nonsplit precisely because $N$ is inert in $\Ocal_c$) . This gives an isomorphism
\[
M_{U_c} (\mathbb{C})\simeq Y_{\ns }(N)_{\mathbb{C}}.
\]

The CM points on $M_{U_c}$ in the sense of Zhang are then the double cosets of pairs $(h_0 ,i_c )$, where $h_0 $ is fixed by the image $T$ of the torus $K^\times $. and $i_c$ has the property that
\[
i_c U_c i_c ^{-1}\cap T(\mathbb{A}_f )\simeq \widehat{\mathcal{O}}^\times _c /\widehat{\mathcal{O}}^\times _F ,
\]
in other words the nonsplit Cartan structure of level $N$ is the one determined by the endomorphism ring of the CM elliptic curve.

On the other hand, we say that $(E,\phi_\varepsilon ) \in Y_{\rm{ns}}(N)$ is a Heegner point (in the sense of Kohen--Pacetti) with multiplication by $\Ocal_c$ if $\End (E) \cong \Ocal_c$  (with $c$ prime to $N$) and $\phi_\varepsilon$ comes from an endomorphism $\beta$ of $E$. Note that this implies that $N$ is inert in $\Ocal_c$, since the minimal polynomial of $\beta $ modulo $N$ is then irreducible.

This discussion thus implies the following equivalence of definitions.

\begin{lemma}
	Under the identification $M_{U_c} \simeq Y_{\ns }(N)$ for every order $\Ocal_c$ of conductor $c$ prime to $N$, Zhang's CM points correspond to Heegner points with CM by $\Ocal_c$ in $Y_{\rm{ns}}(N)$ in the sense of Kohen--Pacetti.
\end{lemma}

Let $f$ be an eigenform in $S_2 (\Gamma _0 (N^2 ))^{+,\new} $. It can be seen as an automorphic form on an $M_{U_c}$ as above, using the isomorphism of Hecke modules $S_2 (\Gamma _0 (N^2 ))^{+,\new} \cong S_2(\Gamma_{\rm{ns}}^+(N))$ and the isomorphism $M_{U_c}(\C) \cong Y_{\rm{ns}}(N)_{\mathbb{C}}$ and we again have by Eichler-Shimura theory a $\Q$-simple quotient $A_f$ of $J_{\rm{ns}}^+(N)$.

The consequence of Zhang's result that we will use is the following.

\begin{Theorem}[\cite{zhang2}, Theorem 6.1]
With notation as above, let $1_c$ be the trivial character of $\Gal (H_c /K)$ and $P$ a Heegner point on $Y_{\rm{ns}}(N)$ with CM by $\Ocal_c$ in the sense of Kohen-Pacetti. Denote by  $P_{1_c}$ be the projection of $P - \xi $ ($\xi$ the Hodge class) in $J_{\ns}(N)(K) = J_{\ns }(N)(H_c )^{1_c}$. Let $P_{1_c}^f$ be the projection of $P_{1_c}$ onto the $f$-isotypical component of $J_{\ns }(N)(H_c )\otimes \mathbb{C}$. 

If $L'(f,1)=0$, then $P_{1_c}^f=0$ and $\pi_f (P_{1_c})$ is torsion in $A_f(H_c)$.
\end{Theorem}

\begin{proof}
	Using the previous lemmas and discussion, we can translate everything in terms of the Shimura curve $M_{U_c}$: the Heegner point $P$ becomes a CM point in the sense of Zhang and $f$ becomes an automorphic representation $\phi$. These changes are compatible with Hecke operators and Galois actions, so they preserve the decompositions into isotypical components above. We can then proceed along the same lines as the proof of Lemma \ref{lemmaHeegnerpoints} part 2 to deduce the conclusion from Zhang's theorem.
\end{proof}

We are now ready to prove the analogue of Lemma \ref{diagonalHeegner} with $X_0 (N)$ replaced by $X_{\ns }^+(N)$.
\begin{lemma}\label{ns_analogue}
Let $X=X_{\ns } (N)$, let $m$ be prime to $N$, and let $\widetilde{C}_m$ be the Hecke correspondence defined above. Then the divisor $i_{1,2} ^* \widetilde{C}_{m}$ is supported on Heegner points in the sense of Kohen-Pacetti and cusps whenever $m$ is less than $N^2 /4$.
\end{lemma}
\begin{proof}
By the moduli interpretation of $X_{\rm{ns}}(N)$ and the Hecke correspondences, a noncuspidal point in the support of $i_{1,2} ^* \widetilde{C}_{m}$ is a pair $(E,\phi_\varepsilon)$ such that there exists an endomorphism $\alpha$ of $E$ of norm $m$ with cyclic kernel (of order $m$) such that if $\overline{\alpha}$ is the induced endomorphism of $E[N]$, $\overline{\alpha} \circ \phi_\varepsilon \circ \overline{\alpha}^{-1} =  \phi_\varepsilon$. This implies that $\overline{\alpha}$ belongs to the nonsplit Cartan subgroup associated to $\phi_\varepsilon$ (which is also the group of invertible elements of $\Z[\phi_\varepsilon]$). We claim $\overline{\alpha }$ is not scalar: if it were, we could write $\alpha = k + N \beta, k \in \Z\beta \in \End(E)$ and then the norm of $\alpha$ being $m <N^2/4$ forces $\beta$ to be an integer as well, contradicting the assumption that $\alpha$ has cyclic kernel.

From this, we deduce that $\Z[\overline{\alpha}] = \Z[\phi_\varepsilon]$, as both are $\Z/N\Z$-vector spaces of dimension 2 and the former is included in the latter. This implies that $\phi_\varepsilon$ is induced by the action of an element of $\Z[\alpha] \subset \End(E)$ on $E[N]$, and the ring of endomorphisms has conductor prime to $N$ for the same reasons as in $X_0(N)$, and its discriminant is automatically prime to $N$ as discussed after defining Heegner points in the sense of Kohen-Pacetti. 

\end{proof}

By the compatibility with Hecke correspondences on $X_0 (N^2 )$ (which is a consequence of Chen's theorem without quotient by Atkin-Lehner involutions, e.g. \cite[Théorème 2]{DeSmitEdixhoven}), Lemma \ref{span_endomorphisms} implies that any nice correspondence $Z$ on $X_{\rm{ns}}^+(N)$ can be written as a linear combination of $C_m$ for $m<N^2 /4$ prime to $N$. By Lemma \ref{ns_analogue}, for any such $Z$, $D_Z (b)$ is supported on Heegner points (in the sense of Kohen--Pacetti) and cusps. Hence, Zhang's Gross--Zagier theorem (together with Manin--Drinfeld) implies $\pi _B (D_Z (b))$ is torsion. 
Assuming the conclusions of Theorem \ref{mainthm} hold for $M$, the Heegner quotient $A$ of $J_0^+(M)^{\rm{new}}$ is of dimension at least 2 so $\rho(A) \geq 2$. This completes the proof of case (2) of Proposition \ref{mainproposition}.

\subsection{An alternative approach}\label{alternative}
In this subsection, we sketch an alternative and less ad hoc approach for proving Proposition \ref{mainproposition} in the case $X=X_0 ^+ (N)$, using the Theorem of Yuan--Zhang--Zhang on the heights of diagonal cycles. 

\begin{Theorem}[Darmon--Rotger--Sols \cite{DRS}, Theorem 3.7]\label{DRS_thm}
Let $X=X_0 (N)$, and let $f,g$ be non-conjugate eigenforms in $S_2 (\Gamma _0 (N))$. Let $Z\in \NS (J_0 (N))$ lie in the image of $\NS (A_g )$.
Suppose $\epsilon (f)=-1$ and $\epsilon (\sym ^2 (g)\otimes f)=1$. If the projection of $D_Z (b)$ to $A_f$ is non-torsion, then $L'(f,1)\neq 0$.
\end{Theorem}
The result above holds for arbitrary $N$, but is most useful when $N$ is prime, since in this case we have $\epsilon (f\otimes g\otimes g)=-a_N (f)a_N (g)^2 =-a_N (f)$ (see e.g. \cite{GK}). Hence in this case Theorem \ref{DRS_thm} implies that the image of $D_Z (b)$ in $A_f$ is torsion for all eigenforms $f$ in $S_2 ^+ (\Gamma _0 (N))$., which implies that  we get an alternative proof for $X_0 ^+ (N)$. One way to view Proposition \ref{mainproposition} is that it shows that it is easier to prove diagonal cycles are torsion than it is to prove they are non-torsion. On the other hand, one can show directly that the image of $D_Z (b)$ in $A_f $ is torsion for all eigenforms $f$ satisfying $w_N (f)=-f$, as explained in \cite[Theorem 3.3.8]{daub}: by Lemma \ref{functoriality_of_diagonal_cycles}, we have 
\[
w_{N}^* (D_Z (b))=D_{w_N ^* (Z)}(b).
\]
Since $w_N ^* (Z)=Z$, and $w_N ^* $ acts as (-1) on $A_f$, we deduce $\pi _{f *}(D_Z (b))$ is torsion.
\section{Proof of the analytic part}\label{analytic_part}

	In this section, we prove Theorem \ref{mainthm} using analytic weighted averages techniques, following guiding principles e.g. from \cite{IwaniecSarnak00} and \cite{Ellenberg04}. For convenience and consistency, the notation below is as close as possible to those from \cite{LeFourn2}.

\textbf{Notation} 

\begin{itemize}
		
	\item 	$N$ is a prime number and $M = N$ or $N^2$ in all of the following.
	
	\item  If $f,g \in S_2(\Gamma_0(M))$, we denote their Petersson scalar product by 
	\[
	\langle f,g \rangle_M = \int_\Dcal \overline{f(x+iy)} g(x+iy) dx dy,
	\]
	where $\Dcal$ is a fundamental domain of $\Gamma_0(M)$, and the associated Petersson norm by $\| \cdot\|_M$.
	
	\item For $\varepsilon= \pm 1$, the space $S_2(\Gamma_0(M))^\varepsilon$ refers to the subspace of modular forms $f$ of $S_2(\Gamma_0(M))$ such that $f_{|w_M} = \varepsilon \cdot  f$, where $w_M$ is the Fricke involution of $S_2(\Gamma_0(M))$. Note that in weight 2, this is the space of modular forms $f$ such that $L(f,s)$ has root number $- \varepsilon$.
	
	\item For $A,B$ linear forms on $S_2(\Gamma_0(M))$ (resp. on a subspace indicated by superscripts), we write 
	\[
	\langle A,B \rangle_M = \sum_f \frac{\overline{A(f)} B(f)}{\|f\|_M^2},
	\]
	where $f$ goes through an orthogonal basis of $S_2(\Gamma_0(M))$ (it is readily checked not to depend on this choice of basis), resp. of the prescribed subspace. We will add superscripts $\{+,-,\rm{new},\rm{old}\}$ to refer to the sum restricted to an orthogonal basis of the corresponding subspaces of $S_2(\Gamma_0(M))$.
	
	\item We denote by $a_m$ (for $m \in \N_{\geq 1}$) and $L'$ the linear forms on $S_2(\Gamma_0(M))$ which to $f$ associate respectively the $m$-th coefficient of the $q$-expansion of $f$, and $L'(f,1)$ (defined properly in the next paragraph).
	
	\item The (positive) greatest common divisors of integers $a,b$ or integers $a,b,c$ are respectively denoted by $(a,b)$ and $(a,b,c)$.
	
	\item For any positive number $B$, $O_1(B)$ refers to a complex number of absolute value  $\leq B$.
	
\end{itemize}

The proof of Theorem \ref{mainthm} relies on the following lemma.

\begin{lemma}
	\label{lemtrickratio}
	Theorem \ref{mainthm} holds for $M$ if
	\[
	\langle a_1, L' \rangle_M^{+,\rm{new}} \neq 0 \quad \textrm{and} \quad \frac{\langle a_2, L' \rangle_M^{+,\textrm{new}}}{\langle a_1,L' \rangle_M^{+, \rm{new}}} \in ]0,1[.
	\]
\end{lemma}

\begin{proof}
	If $\langle a_1,L'\rangle _M^{+,\textrm{new}} \neq 0$, by definition of this sum, there must be at least one normalised newform $f \in S_2(\Gamma_0(M))^{+,\textrm{new}}$ such that $L'(f,1) \neq 0$. As a byproduct of the Gross--Zagier formula (\cite{GrossZagier86}, Corollary V.1.3), this implies that $L'(g,1) \neq 0$ for all normalised newforms $g$ which are conjugates of $f$ by $\GalQ$, thus Theorem \ref{mainthm} holds for $M$ unless the field of coefficients of $f$ is $\Q$ and this $f$ is unique, which we assume now. As $f$ is normalised, those coefficients are algebraic integers hence belong to $\Z$. Now, one has 
	\[
	\frac{\langle a_2, L' \rangle_M^{+,\textrm{new}}}{\langle a_1,L' \rangle_M^{+, \textrm{new}}} = \frac{\overline{a_2(f)} L'(f,1) \|f\|_M^2}{\overline{a_1(f)} L'(f,1) \|f\|_M^2} = a_2(f) \in ]0,1[
	\]
	by hypothesis, so $a_2(f) \notin \Z$ which leads to a contradiction and Theorem \ref{mainthm} holds.
\end{proof}

\begin{Remark}
	\label{remlemmatrickratio}
	The statement of this lemma appears quite \textit{ad hoc} so let us explain the main motivations behind it.
	
	\begin{itemize}
		\item As we will see later, as long as $m$ is small compared to $\sqrt{M}$, one has 
		\[
		\frac{ \langle a_m,L' \rangle_{M}^{+,\textrm{new}}}{4 \pi} = \ln(\sqrt{M}) + C - \ln(m) + O(m/\sqrt{M})
		\] with explicit implied constants. This proves that the hypotheses of the lemma are indeed satisfied for large $M$.
		
		\item The error terms of the estimate above are smaller when the $m$'s are smaller, hence the choices of $m=1$ and 2 for the ratio.
		
		\item There are far better asymptotic estimates on the number of newforms $f$ in $S_2(\Gamma_0(M))^{+,\textrm{new}}$ such that $L'(f,1) \neq 0$, e.g. : by \cite{KowalskiMichelVanderKam} (at least for $M=N$ prime), the proportion of such forms is asymptotically at least $7/8$, in particular there are far more than just 2 for $M$ large). These techniques, using also estimates of second moments and of the norms $\|f\|_M$, are harder to make explicit, and we suspect the effective bounds obtained by following step-by-step the arguments would be huge. Lemma \ref{lemtrickratio}, while very crude (and giving a weaker result) is tailor-made to be efficient enough for precise estimates and approachable bounds.
	\end{itemize}

\end{Remark}

\subsection{Splitting of the terms to estimate the first moments}

The starting point to estimate the weighted averages $\langle a_m, L'\rangle_N^{\rm{new}}$ is the following trace formula of Petersson adapted by Akbary (and proven in greater generality in \cite{LeFourn2}).

\begin{Proposition}
	\label{propformulesdestraces}
	\hspace*{\fill}
	
	Let $m,n,M$ be three positive integers, and $\varepsilon = \pm 1$. Then, we have
	
	\begin{eqnarray}
	\label{formuledestracesgeneralisee}
	\frac{1}{2 \pi \sqrt{mn}} \langle a_m, a_n \rangle_{M}^\varepsilon   =  \delta_{mn} &  - &  2 \pi \sum_{\substack{c >0 \\ M |c} } \frac{S(m,n ;c)}{c} J_1 \left( \frac{4 \pi \sqrt{mn}}{c} \right)
	\\ & -  & 2 \pi \varepsilon \sum_{\substack{d > 0  \\ (d,M)= 1}} \frac{ S(m,nM^{-1};d)}{d \sqrt{M}} J_1 \left( \frac{4 \pi \sqrt{mn}}{d \sqrt{M}} \right) \nonumber,
	\end{eqnarray}
	where $S$ is the notation for Kloosterman sums
	\[
	S(m,n;c) = \sum_{k \in (\Z /c \Z)^*} e^{ 2 i \pi (m k + n k^{-1})/c}
	\]
	(except for $c=1$ where its value is 1 by convention), $Q^{-1}$ means the inverse of $Q$ modulo $d$ in the Kloosterman sums and $J_1$ is the Bessel function of the first kind and order 1.
\end{Proposition}

The sums on the right-hand side are absolutely convergent thanks to the following well-known uniform bounds: $|J_1(x)| \leq |x|/2$ for all $x$, and the Weil bounds
\begin{equation}
\label{eqbornesWeil}
|S(m,n;c)| \leq (m,n,c) ^{1/2} \tau(c) \sqrt{c},
\end{equation}
with $\tau$ the divisor-counting function, which improves, if $M$ is a prime power dividing $c$, in
\[
|S(m,n;c)| \leq  2 (m,n,c) ^{1/2} \tau(c/M) \sqrt{c}
\]
(\cite{IwaniecKowalski}, (3.2), (3.3), Theorem 11.11 and Corollary 11.12).

Now, our normalisation of the L-function associated to a form $f \in S_2(\Gamma_0(M))$ is given by 
\[
L(f,s) = \sum_{n=1} \frac{a_n(f)}{n^s},
\]
and this L-series converges uniformly on any compact subset of $\{\Re(s)>2 \}$.

One can express $L'(f,1)$ itself in terms of the Fourier coefficients of $f$ in the following way.

\begin{lemma}
	\label{lemvaluederLat1}
	For any $M \geq 1$ and any $f \in S_2(\Gamma_0(M))^+$, one has 
	\[
	L'(f,1) = 2 \sum_{n=1}^{+ \infty} \frac{a_n(f)}{n} E_1 \left( \frac{2 \pi n}{\sqrt{M}} \right)
	\]
	where $E_1$ is the exponential integral function, defined on $]0,+\infty[$ by 
	\[
	E_1(y) = \int_y^{+ \infty} \frac{e^{-t}}{t}dt.
	\]
\end{lemma}

\begin{proof}
	We define the completed L-function $\Lambda$ associated to $L$ by 
	\begin{equation}
	\label{eqdefLambda}
	\Lambda(f,s) := \left( \frac{\sqrt{M}}{2 \pi} \right)^{s} \Gamma(s) L(f,s).
	\end{equation}
	By standard arguments(e.g. \cite{Bump}, section 1.5), this function extends to an holomorphic function on $\C$ and satisfies the functional equation
	\begin{equation}
	\label{eqeqfoncLambda}
	\Lambda(f,2-s) = - \Lambda(f_{|w_M},s).
	\end{equation}
	The expression of $L'(f,1)$ is then deduced from the functional equation of $\Lambda$ by integration of residues on vertical axes and Mellin transform (see e.g. \cite{IwaniecKowalski} (26.10) where the definition of $L$ is translated by $1/2$).
\end{proof}

With this formula and by uniform convergence of the terms involved, we obtain:

\begin{equation}
\label{eqdevmoment}
\frac{\langle a_m, L' \rangle_{M}^+}{4 \pi} = E_1 \left( \frac{2 \pi m}{\sqrt{M}} \right) - 2 \pi \sqrt{m} \left( \sum_{M|c} \frac{\Scal(c)}{c} + \sum_{(d,M)=1} \frac{\Tcal(d)}{d \sqrt{M}} \right),
\end{equation}
where
\begin{equation}
\label{eqdefScal}
\Scal(c) = \sum_{n=1}^{+ \infty} \frac{S(m,n;c)}{\sqrt{n}} J_1 \left( \frac{4 \pi \sqrt{mn}}{c} \right) E_1 \left( \frac{2 \pi n}{\sqrt{M}} \right)
\end{equation}
and
\begin{equation}
\label{eqdefTcal}
\Tcal(d) = \sum_{n=1}^{+ \infty} \frac{S(m,nM^{-1};d)}{\sqrt{n}} J_1 \left( \frac{4 \pi \sqrt{mn}}{d \sqrt{M}} \right) E_1 \left( \frac{2 \pi n}{\sqrt{M}}
\right). 
\end{equation}

The main term in \eqref{eqdevmoment} will be $E_1(2 \pi m /\sqrt{M})$ as long as $m \ll \sqrt{M}$.

The trace formula does not separate the old and new spaces, which we need for $M=N^2$. This is taken care of in the following lemma.

\begin{lemma}
	For $N$ prime and $m \geq 1$ not divisible by $N$, 
	\[
	\langle a_m, L' \rangle_{N^2}^{+,\rm{new}} = \langle a_m, L' \rangle_{N^2}^+ - \frac{1}{N-1} \left( \langle a_m, L' \rangle_N^+ + \frac{\ln(N)}{2} \langle a_m, L \rangle_N^- \right).
	\]
\end{lemma}

\begin{proof}
	By orthogonality of the new and old subspaces,
	\[
	\langle a_m, L' \rangle_{N^2}^{+,\textrm{new}}  =  \langle a_m, L' \rangle_{N^2} -  \langle a_m, L' \rangle_{N^2}^{+,\textrm{old}}.
	\]
	
	To prove the formula on the oldpart, we need to be a bit careful with the definitions of completed L-functions: although the definition of $L(f,s)$ does not depend on the ambient space of modular forms, the definition of the completed L-function $\Lambda(f,s)$ in \eqref{eqdefLambda} does.  The degeneracy operators are denoted by $A_n$ as in the original article \cite{AtkinLehner70}. Let
	\[
	A_1 = I_2, \quad A_N = \begin{pmatrix} N & 0 \\ 0 & 1 \end{pmatrix}, \quad W_N =  \begin{pmatrix} 0 & 1 \\ -N & 0 \end{pmatrix}, \quad W_{N^2} =  \begin{pmatrix} 0 & 1 \\ -N^2 & 0 \end{pmatrix}.
	\]
	Notice that $(A_N W_{N^2} W_N^{-1})/N$ belongs to $\Gamma_0(N)$, thus for $f \in S_2(\Gamma_0(N))$ such that $f_{|W_N} = \varepsilon_f \cdot f$, one has 
	\begin{equation}
	\label{eqAtkinLehneretdegenerescence}
	(f_{|A_N})_{|W_{N^2}} = (f_{|W_N})_{|A_1} = \varepsilon_f \cdot f_{|A_1},
	\end{equation}
	hence also
	\[
	(f_{|A_1})_{|W_{N^2}} = \varepsilon_f \cdot f_{|A_N}.
	\]
	Consequently, an orthogonal (see the computations of section 4 of \cite{LeFourn2} for example) basis of $S_2(\Gamma_0(N^2))^{+,\textrm{old}}$ is given by the $f_{|A_1} + (f_{|A_1})_{|W_{N^2}}$, where $f$ goes through an eigenbasis of $S_2(\Gamma_0(N))$. The aforementioned computations also prove with \eqref{eqAtkinLehneretdegenerescence} that if $f_{|W_N} = \varepsilon_f \cdot f$, then
	\[
	\langle f_{|A_1} + (f_{|A_1})_{|W_{N^2}},f_{|A_1} + (f_{|A_1})_{|W_{N^2}} \rangle_{N^2} = 2(N-1) \langle f,f \rangle_N.
	\]
	If $N$ does not divide $m$ (so that $a_m(f_{|A_N})=0$), this implies that
	\[
	\langle a_m, L' \rangle_{N^2}^{+,\textrm{old}} = \frac{1}{2(N-1)} \sum_f \overline{a_m(f)} L'( f_{|A_1} + (f_{|A_1})_{|W_{N^2}},1)
	\]
	where $f$ goes through an orthonormal basis of $S_2(\Gamma_0(N))$. Now, by the functional equation of $\Lambda(f,s)$ in \eqref{eqeqfoncLambda}, $	\Lambda'(f_{|A_1},1) = \Lambda'((f_{|A_1})_{|W_{N^2}},1)$ but
	\begin{eqnarray*}
		\Lambda'(f_{|A_1},1) & = & \frac{N}{2 \pi} (L'(f_{|A_1},1) + (\ln(N/2\pi) + \gamma) L(f,1)) \\
		\Lambda'((f_{|A_1})_{|W_{N^2}},1)&  = & \frac{N}{2 \pi} (L'((f_{|A_1})_{|W_{N^2}},1) + (\ln(N/2\pi) + \gamma) \varepsilon_f L(f,1)).
	\end{eqnarray*}
	
	The first equality is a direct application of the definition of $\Lambda$, the second one uses that $L(f_{|A_N},1) = L(f,1)$ (easy to show by the integral formula of $L(f,1)$) and the results above. Thus, to compute  $L'(f_{|A_1} + (f_{|A_1})_{|W_{N^2}},1)$, it is enough to know the sum of the two right-hand terms which is the sum of the two left-hand terms, which equal one another. Now, if $\varepsilon_f=1$ then $L(f,1)=0$ by sign of the functional equation of $\Lambda(f,s)$ (in level $N$ here !), and if $\varepsilon_f = -1$, $\Lambda'(f,1) =0$. We thus obtain in this case
	\[
	L'(f,1) = - (\ln(\sqrt{N}/(2 \pi)) + \gamma) L(f,1).
	\]
	and get the lemma by summation on those forms $f$'s gathered by sign of $\varepsilon_f$.
\end{proof}       

\subsection{First estimates}

We recall that $M=N$ or $N^2$.

\begin{lemma}
	\label{lemWeilbounds}
	
	Using the Weil bounds, we get for every $c$ multiple of $M$ and $d$ prime to $M$: 
	\[
		|\Scal(c)|  \leq  2 \sqrt{mM} \tau(c/M)  \frac{f((m,c))}{\sqrt{c}}, \quad 
	|\Tcal(d)|  \leq  \tau(d) \sqrt{m} \frac{f((m,d))}{\sqrt{d}}
	\]
	where for every integer $k$, $f(k) = \sum_{k'|k} \frac{1}{\sqrt{k'}}$. 	For $m=2$ and $c$, $d$ even, these estimates are improved to
	\begin{equation}
	\label{eqboundsScalTcalWeil}
	|\Scal(c)|  \leq  (\sqrt{2}+2) \frac{\sqrt{M} \tau(c/M)}{\sqrt{c}}, \quad 
	|\Tcal(d)|  \leq  (1+1/\sqrt{2}) \frac{\tau(d)}{\sqrt{d}}.
	\end{equation}
\end{lemma}

\begin{proof}
	In the definitions of $\Scal(c)$ (and similarly for $\Tcal(d)$), we separate the terms in $n$ depending on the values of $(m,n,c) = m'$ which is a divisor of $(m,c)$. Then, using $|J_1(x)| \leq |x|/2$, it only remains to control the sum of the $E_1(2 \pi m'n/\sqrt{M})$ for $n$ from 1 to $+ \infty$, which after sum-integral comparison and variable change is smaller than $\sqrt{M}/(2\pi m')$.
	
	In the specific case where $m=2$ and $c$ or $d$ even, the cases are made from the beginning on the values of $(m,n,c)^{1/2}$ instead of bounding by $(m,c)^{1/2}$, and a careful computation gives those bounds.
\end{proof}

This allows to bound the sum of the $\Scal(c)/c$ for all multiples $c$ of $M$. By multiplicativity of $\tau$, 
\begin{eqnarray*}
	\left| \sum_{M|c} \frac{\Scal(c)}{c} \right| & \leq & \frac{2 \sqrt{m}}{M} \sum_{m'|m} \frac{f(m') \tau(m')}{(m')^{3/2}} \sum_{c=1}^{+ \infty} \frac{\tau(c)}{c^{3/2}} \\
	& \leq & \frac{2 \sqrt{m}}{M} \sum_{m'|m} \frac{\tau(m')}{m'} \sum_{c=1}^{+ \infty}  \frac{\tau(c) }{c^{3/2}},
\end{eqnarray*}
the sum on $c$ being exactly $\zeta(3/2)^2$. We denote 
\[
g(m) = \sum_{m'|m} \frac{f(m') \tau(m')}{(m')^{3/2}}.
\]
hence (and similarly for $\Tcal$): 
\begin{equation}
\label{eqineq}
2 \pi \sqrt{m} \left| \sum_{M|c} \frac{\Scal(c)}{c} \right|  \leq  \frac{86 m}{M} g(m), \quad 
2 \pi\sqrt{m} \left| \sum_{(d,M)=1} \frac{\Tcal(d)}{d\sqrt{M}} \right|  \leq  \frac{43 m}{\sqrt{M}} g(m)
\end{equation}
which gives 
\begin{equation}
\label{eqpremiereestimation}
\frac{\langle a_m, L' \rangle_M^+}{4 \pi} = E_1 (2 \pi m / \sqrt{M}) + g(m) m \left( O_1 \left(\frac{86}{M} \right) + O_1 \left(\frac{43 }{\sqrt{M}} \right) \right).
\end{equation}
For $m=2$, the previous refinements can be exploited and we get instead
\[
	2 \pi \sqrt{2} \left| \sum_{M|c} \frac{\Scal(c)}{c} \right|  \leq  \frac{213}{M},\quad 
	2 \pi \sqrt{2}  \left| \sum_{(d,M)=1} \frac{\Tcal(d)}{d \sqrt{M}} \right| \leq  \frac{97}{\sqrt{M}}
\]
hence
\begin{equation}
\label{eqestimraffineesmegaldeux}
\frac{\langle a_2, L' \rangle_M^+}{4 \pi} = E_1 (4 \pi / \sqrt{M}) + O_1 \left( \frac{213}{M} \right) + O_1 \left( \frac{97}{\sqrt{M}} \right).
\end{equation}
Identical bounds are found for 
\begin{eqnarray*}
	\Scal_0(c) & = & \sum_{n=1}^{+ \infty} \frac{S(m,n;c)}{\sqrt{n}} J_1 \left( \frac{4 \pi \sqrt{mn}}{c} \right) \exp \left( - \frac{2 \pi n}{\sqrt{M}} \right) \\
	\Tcal_0(d) & = & \sum_{n=1}^{+ \infty} \frac{S(m,nM^{-1};d)}{\sqrt{n}} J_1 \left( \frac{4 \pi \sqrt{mn}}{c\sqrt{M}} \right) \exp \left( - \frac{2 \pi n}{\sqrt{M}} \right)
\end{eqnarray*}
as the integral of $e^{-t}$ on $[0,+\infty[$ is equal to 1 like the one of $E_1$. Thus, by similar computations,
\[
\frac{\langle a_m,L \rangle_N^-}{4 \pi} = e^{ - 2 \pi m/\sqrt{N}} + m g(m) \left( O_1 \left( \frac{86}{N} + \frac{43}{\sqrt{N}} \right) \right).
\]
Gathering those bounds, we get for all $m$ prime to $N$,
\begin{eqnarray}
\label{eqestimationp2regroupee}
\frac{\langle a_m,L'\rangle_{N^2}^{+,\textrm{new}}}{4 \pi} & = & E_1\left( \frac{2 \pi m}{N} \right) - \frac{E_1 \left( \frac{2 \pi m}{\sqrt{N}} \right)}{N-1}  - \frac{\ln(N) e^{- 2 \pi m/\sqrt{N}}}{2(N-1)}   \\
& + &  m g(m) O_1 \left( \frac{86}{N^2} + \frac{43}{N} + \frac{\ln(N)/2+1}{N-1} \left(\frac{86}{N} + \frac{43}{\sqrt{N}} \right) \right)
\end{eqnarray}
and slightly better ones for $m=2$ coming from refinements above (it suffices to replace $86mg(m)$ by 213 and $43 mg(m)$ by 97 above).

By computations on Sage, we deduce the following first estimates.

\begin{Proposition}
	With the previous estimates, one finds
	\[
	\begin{array}{rcl|rcl}
	\langle a_1, L' \rangle_{N}^+ > 0 & \textrm{for} & N \geq 1213 & \langle a_1, L' \rangle_{N^2}^{+,\rm{new}} > 0 & \textrm{for} & N \geq 47 \\
	\langle a_2, L' \rangle_{N}^+>0 & \textrm{for} & N \geq 5437 &  \langle a_2, L' \rangle_{N^2}^{+,\rm{new}} > 0 & \textrm{for} & N \geq 97 \\
	\frac{\langle a_2, L' \rangle_{N}^+}{\langle a_1, L' \rangle_{N}^+} \in ]0,1[ & \rm{for} & N \geq 45341 & \frac{\langle a_2, L' \rangle_{N^2}^{+,\rm{new}}}{\langle a_1, L' \rangle_{N^2}^{+,\rm{new}}} \in ]0,1[ & \textrm{for} & N \geq 269.
	\end{array}
	\]
	hence Lemma \ref{lemtrickratio} applies and Theorem 2 is true for $N \geq 45341$ for $X_0 ^+ (N)$ and for $N \geq 269$ for $X_{\rm{ns}}^+(N)$.
\end{Proposition}

For $M=N$, the estimates of $\langle a_m,L'\rangle_N$ are readily obtained, but the slowness of convergence is much more visible. This is mainly due to the fact that the error term is in $m/\sqrt{N}$ instead of $m/N$.

\subsection{Improving the estimates for prime level}

To attain from $N \geq 45341$ a range where all remaining primes can be checked by a different method, one needs to improve upon the worst error term appearing in $\langle a_m,L' \rangle_N^+$, which is in $m/\sqrt{N}$ and comes from the estimates of $\Tcal(d)$ after looking at \eqref{eqboundsScalTcalWeil}.

The following arguments rely on cancellations of Kloosterman sums not exploited by the Weil bounds. For $d=1$, the Kloosterman sum is always 1 (see the convention) so this case has to be dealt with separately. A careful analysis proves that 
\[
0.4 \sqrt{m} \leq \Tcal(1) \leq \sqrt{m},
\]
which will slightly improve the bounds later. 

Assume now that $d \geq 2$. The main term contributing to the bound is $E_1(2\pi n/\sqrt{N})$, hence we write 
\[
\Tcal(d) = \Tcal_M(d) + \Tcal_R(d),
\]
where $\Tcal_M(d)$ is the sum of terms for which $n \leq 3 \sqrt{N}/\pi$ and $\Tcal_R(d)$ is the remainder.

By the Weil bounds, using the fact that the integral of $E_1$ on $[5,+\infty[$ is less than $10^{-4}$, we obtain
\[
2 \pi \sqrt{m} \sum_{d \geq 2} \left| \frac{\Tcal_R(d)}{d \sqrt{N}} \right| \leq 10^{-4} \frac{\lambda_m}{\sqrt{N}}
\]
where $\lambda_m = 43$ for $m=1$ and $97$ for $m=2$ as before, so this contribution will be very small. For $\Tcal_M(d)$, we will exploit Poly\`a-Vinogradov-type estimates (\cite{LeFourn1}, Lemma 5.9).

\begin{Proposition}
	For every $d>1$, every $k$ invertible modulo $d$ and every $m,K,K' \in \N$, 
	\[
	\left| \sum_{n=K}^{K'} S(m,nk;d) \right| \leq \frac{4d}{\pi^2} (\log(d) + 1.5).
	\]
\end{Proposition}	

Now, assume $N \geq 1000$, so that for $m=1$ or 2 and $n \leq 5 \sqrt{N}/(2 \pi)$, $4 \pi \sqrt{mn}/(d \sqrt{N}) \leq 1.5$. This implies that in the considered range for $n$, the function $t \mapsto J_1(4 \pi \sqrt{mt}/(d \sqrt{N}))/\sqrt{t} E_1(2 \pi t/\sqrt{N})$ is decreasing and positive (as the product of two such functions). Its total variation on $[1,5 \sqrt{N}/2 \pi]$ is then bounded by its first value (itself controlled by $E_1(2 \pi/\sqrt{N})/2$).

By Abel transform and the previous proposition, we thus obtain 
\[
|\Tcal_{M}(d)| \leq \frac{8}{\pi} \frac{\sqrt{m}}{\sqrt{N}} (\log(d)+1.5) E_1 \left( \frac{2 \pi}{\sqrt{N}} \right).
\]
Compared to Weil bounds in Lemma \ref{lemWeilbounds}, the new bound is approximately the best for $d \leq f(N)= \lfloor N/(2.5^2 E_1(2\pi/\sqrt{N})^2) \rfloor$. We then obtain
\begin{eqnarray*}
2 \pi \sqrt{m}  \left| \sum_{d=2}^{f(N)} \frac{\Tcal_{M}(d)}{d \sqrt{N}} \right| & \leq & \frac{16 m}{N} E_1 \left( \frac{2 \pi}{\sqrt{N}} \right) \sum_{d=2}^{f(N)} \frac{\log(d)+1.5}{d} \\
& \leq & \frac{8m}{N} E_1 \left( \frac{2 \pi}{\sqrt{N}} \right) \left( \log(f(N))^2 + 3 \log(f(N)) + 1 \right)
\end{eqnarray*}
with lemma 5.11 of \cite{LeFourn1}. By Weil bounds and the same lemma, for $m=1$, 
\begin{equation}
\label{eqestimsommeTmcal1Weil}
2 \pi  \left| \sum_{d=f(N)+1}^{+ \infty} \frac{\Tcal_{M}(d)}{d \sqrt{N}} \right| \leq  \frac{4 \pi}{\sqrt{N f(N)}} (\log(f(N))+4) 
\end{equation}
and for $m=2$, 
\begin{equation}
\label{eqestimsommeTmcal1Weilpourm2}
2 \pi \sqrt{2} \left| \sum_{d=f(N)+1}^{+ \infty} \frac{\Tcal_{M}(d)}{d \sqrt{N}} \right|  \leq \frac{8 \pi (2-1/\sqrt{2})}{\sqrt{N f(N)}} (\log(f(N))+4). 
\end{equation}

Combining these arguments, we get, for $N \geq 1000$,
\[
\frac{\langle a_1,L' \rangle_N^{+}}{4 \pi} \geq E_1 \left( \frac{2 \pi}{\sqrt{N}} \right) - \frac{6.3}{\sqrt{N}} - \frac{86}{N} - 2 \pi \left| \sum_{d=2}^{+ \infty} \frac{\Tcal_M(d)}{d \sqrt{N}} \right| 
\]
and
\[
\frac{\langle a_2,L' \rangle_N^{+}}{4 \pi} \geq E_1 \left( \frac{4 \pi}{\sqrt{N}} \right) - \frac{6.3\sqrt{2}}{\sqrt{N}} - \frac{213}{N} - 2 \pi\sqrt{2} \left| \sum_{d=2}^{+ \infty} \frac{\Tcal_M(d)}{d \sqrt{N}} \right| 
\]

and finally 
\[
\langle a_1,L' \rangle_N^{+} >0 \quad \textrm{and} \quad \frac{\langle a_2,L' \rangle_N^{+}}{\langle a_1,L' \rangle_N^{+}} \in ]0,1[
\]
for $N \geq 8641$, which is much more reasonable than $45341$.

The same improvements for the bounds apply exactly for $M=N^2 \geq 1000$, thus allowing to replace the estimate in $43/N$ in \eqref{eqestimationp2regroupee} by the same expressions as above with $f(M)$ instead of $f(N)$.

One gets that $\langle a_2,L' \rangle_{N^2}^{+,\rm{new}} >0$ for $N \geq 71$ instead of $97$, and that 
\[
\frac{\langle a_2,L' \rangle_{N^2}^{+,\rm{new}}}{\langle a_1,L' \rangle_{N^2}^{+,\rm{new}}} \in ]0,1[
\]
for $N \geq 151$.

We now discuss how to deal with the remaining cases, namely those for which $N \leq 8641$ and $g(X_0^+(N)) \geq 2$, and those for which $N \leq 151$ and $g(X_{\textrm{ns}}^+(N)) \geq 2$. 

The most natural approach is the following: for any small $N$, compute a basis of eigenforms for $S_2(\Gamma_0(M))^{+,\textrm{new}}$, and for every $f$ (normalised) in this basis, compute $L'(f,1)$ up to sufficient precision to ensure that $L'(f,1) \neq 0$. 

Recall that by (\cite{GrossZagier86}, Corollary V.1.3), if $L'(f,1) \neq 0$ under the same assumptions, the same is true for the Galois conjugate eigenforms, so only one check needs to be performed for the Galois orbit. Theorem \ref{mainthm} requires exactly that the sum of sizes of those Galois orbits is at least 2, so we only need to check that for two Galois orbits of size 1 (or one of size at least 2), one has $L'(f,1) \neq 0$. 

We have performed these verifications in MAGMA, and obtained that :

$\bullet$ For any prime $N \leq 2000$ such that $X_0 ^+ (N)$ is of genus at least two, there are at least two distincts normalised newforms such that $L'(f,1) \neq 0$, hence Theorem 2 holds. In fact, we have also checked that for all such $N$, $L'(f,1) \neq 0$ for \textit{all} the eigenforms in $S_2(\Gamma_0(N))^{+}$, therefore by Proposition \ref{propGZK}, $\rank J_0^+(N) (\Q) = \dim J_0^+(N)$ unconditionally for all those small primes.

$\bullet$ Similarly, for any prime $N \leq 53$ such that $X_{\rm{ns}}^+(N)$ is of genus at least two,
$L'(f,1) \neq 0$ for \textit{all} the eigenforms in $S_2(\Gamma_0(N^2))^{+,\textrm{new}}$, therefore by the same arguments, $\rank \operatorname{Jac} (X_{\rm{ns}}^+(N))(\Q) = \dim \operatorname{Jac} (X_{\rm{ns}}^+(N))$ for all those small primes.

Unfortunately, these algorithms require explicit embeddings of the fields of coefficients $K_f$ of $f$ into $\C$, which makes them very slow when $N$ becomes larger than 2000 (then, the degree of $K_f$ can be larger than 100). We thus could not complete the argument by using only this method, let us explain how to deal with the intermediary range $N \in [2000,9000]$ for $X_0^+(N)$ and $N \in [59,151]$ for $X_{\rm{ns}}^+(N)$.

The idea is to look at the simple quotients of the two relevant Jacobians which are elliptic curves. If there are none, in this range, we have proved that $\langle a_1,L' \rangle_M^{+,\textrm{new}} \neq 0$ so we must have $f$ such that $L'(f,1) \neq 0$, and it generates a simple quotient of dimension at least 2 by hypothesis, so we are done. 

Now, if there \textit{are} elliptic curves in there, it is sufficient to find two of them of rank 1 for the same reasons. Quotients of $J_0(M)^{+,\textrm{new}}$ of dimension 1 are in one-to-one correspondence with isogeny classes of elliptic curves of conductor $N$ and root number $-1$ (the fact that this correspondence is surjective is a consequence of Cremona's tables in this range but also a particular case of modularity theorems).

One can thus eliminate all levels $N$ except the ones for which there exists exactly one (up to isogeny) elliptic curve $E$ of analytic rank 1 and conductor $N$. Using Cremona's tables, we obtain a list of respectively 70 ($M=N$) and 7 ($M=N^2$) possible exceptions, namely $N$ in $\{61,67,73,101,109,113\}$ for the latter.

Now, we use a last argument: if the modular form $f_E$ associated to $E$ is really the only one such that $L'(f,1) \neq 0$ in the space, one should have
\[
\langle a_1,L' \rangle_{M}^{+,\textrm{new}} = \frac{L'(E,1)}{\|f_E\|^2}
\]
(the fact that this equality holds without a normalisation factor comes from the Manin constant being equal to 1 here, which is true in this range by results of Cremona).

Now, the left-hand side is larger than $4/5$ for $M=N$, $N \geq 2000$ and than $1/2$ for $M=N^2$, $N \geq 53$ by the (optimised) lower bounds given above, and the right-hand side is computable in terms of periods of $E$. Using this idea turns out to eliminate all remaining possible exceptions in both cases of $M$, which concludes the proof.

\begin{Remark}
	In some sense, this heuristic is natural: all terms in the sum defined by $\langle a_1,L' \rangle_{M}^{+,\textrm{new}}$ are positive (another consequence of Gross--Zagier formula), hence there is no cancellation among those, and the idea is that one of them alone cannot be enough to approach the estimates given for the sum.
\end{Remark}
\section{Appendix: Chow--Heegner points and Ceresa cycles}
\label{AppendixChowHeegnerCeresa}

In this appendix we explain how Lemma \ref{DZB_HM} is a consequence of Hain and Matsumoto's work relating the extension $[\lie (U_2 )]$ to the Ceresa cycle.
\subsection{Ceresa cycles and Gross--Kudla--Schoen cycles}\label{CH_Z}
We recall some properties of modified diagonal cycles studied in\cite{gross-schoen}, \cite{CvG} and \cite{DRS}. As our discussion applies in fairly broad generality, we take $X$ to be a smooth geometrically irreducible projective curve over a field $K$ of characteristic zero.
Let $\pi _S$ denote the projection
\[
X^n \to X^{\# S}
\]
defined by projecting onto the coordinates in $S$ as in \eqref{eqdefpiS}. The \textit{Gross--Kudla--Schoen} cycle is defined to be
\[
\Delta _{GKS}:=\sum _{\emptyset \neq S \subset \{ 1,2,3\}}(-1)^{\# S-1}X_S ,
\]
where $X_S $ is as defined in section \ref{subsecChowHeegnerdiagonalcycles}.

It defines an element of the group $\CH ^2 (X^3 )$ of codimension two cycles in the triple product $X\times X \times X$. By \cite[Proposition 3.1]{gross-schoen}, the class of $\Delta _{GKS}$ lies in the subspace $\CH ^2 _0 (X^3 )$ of homologically trivial cycles.

Now let $Z\subset X\times X$ be a correspondence, and let 
\[
\Pi _Z :\CH^2 (X^3 )\to \CH^1 (X)
\]
be the composite map
\[
\CH^2 (X^3 )\stackrel{\pi _{\{1,2,3\}}^* }{\longrightarrow }\CH^2 (X^4 )\stackrel{\cdot (Z \times X^2 )}{\longrightarrow } \CH^4 (X^4 ) \stackrel{(\pi_4)_* }{\longrightarrow }\CH^1 (X),
\]
where the second map is the intersection product with $Z \times X^2 \subset X^4 $.
\begin{lemma}[\cite{DRS} Lemma 2.1]\label{drs-z}
We have
\[
D_Z (b)=\Pi _Z (\Delta _{GKS}).
\]
\end{lemma}

\subsection{The Gross--Kudla--Schoen cycle and the Ceresa cycle}
Since $[\Delta _{GKS}]$ is homologically trivial, it has (\S \ref{subsecremindersNeronSeveri}) an \'etale Abel-Jacobi class
\[
\AJ _{\et}([\Delta _{GKS}]) \in H^1 (G_K ,H^3 _{\et}(X^3 _{\overline{K}},\Q _p (2))).
\]
By \cite[Corollary 2.6]{gross-schoen}, the cycle class $\AJ _{\et }([\Delta _{GKS}])$ lies in the image of the Kunneth projector 
\begin{align*}
P_{e*}:H^1 (G_K ,H^3 _{\et}(X^3 _{\overline{K}},\Q _p (2))) & \to H^1 (G_K ,H^1 _{\et }(X_{\overline{K}},\Q _p )^{\otimes 3}(2)) \\
& \simeq H^1 (G_K ,V^{\otimes 3}(-1)) \\
& \hookrightarrow H^1 (G_K ,H^3 _{\et}(X^3 _{\overline{K}},\Q _p (2))),
\end{align*}
and hence may be thought of as an element of $H^1 (G_K ,V^{\otimes 3}(-1))$ (here $V:=H^1 _{\et }(X_{\overline{K}},\Q _p (1))$).
The action of $S_3 $ on $X^3 $ induces an action on $V^{\otimes 3}(-1)$, which is given by $\epsilon \otimes \sigma $, where $\epsilon $ is the sign of a permutation and $\sigma $ is the natural action of $S_3 $ on $V^{\otimes 3}$. Since $\Delta _{GKS}$ is invariant under the $S_3 $ action, it lies in the image of $H^1 (G_K ,\wedge ^3 V (-1))$ under the map induced by the inclusion
\begin{align}
\iota :\wedge ^3 V & \to V^{\otimes 3} \label{iota} \\
v_1 \wedge v_2 \wedge v_3 & \mapsto \frac{1}{6}\sum _{\tau \in S_3 }\epsilon (\tau )v_{\tau (1)}\otimes  v_{\tau (2)}\otimes v_{\tau (3)}. \nonumber
\end{align}

For the relations to fundamental groups, it will be helpful to recall the relation between $\Delta _{GKS}$ and the \textit{Ceresa cycle}.
By \cite[Proposition 5.3]{gross-schoen}, the image of $\Delta _{GKS}$ in $\CH ^{g-1}(J)$ under the map
\begin{align*}
\mu :X^3 \to J \\
(x_i ) \mapsto \sum [x_i ] -3[b]
\end{align*} 
is rationally equivalent to
\[
([3]_* -3[2]_* +3[1]* -3[0]_* )\AJ (X).
\] 
The \textit{Ceresa cycle} $C_b$ is defined to be
\[
\AJ (X)-[-1]_* \AJ (X)\in \CH ^{g-1}(J).
\]
\begin{Proposition}[Colombo--van Geemen,\cite{CvG}, Proposition 2.9]\label{cvg}
We have
\[
\AJ _{\et }(\mu _* (\Delta _{GKS}))=3\AJ _{\et }([C_b ])
\]
in $H^1 (G_K ,\wedge ^3 V(-1)).$
\end{Proposition}

We first recall Hain and Matsumoto's description of the Galois action on $U_2$. We again take $X$ to be a smooth projective geometrically irreducible curve over a field $K$ of characteristic zero.
The group $U_2 $ is an extension
\begin{equation}\label{U2_extnHainMatsumoto}
1 \to \Ker (H^2 (J_{\overline{\Q }},\Q _p )\stackrel{\AJ ^* }{\longrightarrow }H^2 (X_{\overline{\Q }},\Q _p ))^* \to U_2 \to V \to 1.
\end{equation}
with $V = T_p J \otimes \Q_p$ again. We define 
\[
\overline{\wedge ^2 V}:=\Ker (H^2 (X_{\overline{\Q }},\Q _p )\stackrel{\AJ ^* }{\longrightarrow }H^2 (J_{\overline{\Q }},\Q _p ))^* ,
\]
and write the image of $v_1 \wedge v_2 $ in $\overline{\wedge ^2 V}$ as $\overline{v_1 \wedge v_2 }$.
Taking the Lie algebra $L_2 $ of $U_2$, we obtain an element $[L_2 ]\in \ext ^1 _{G_K }(V,\overline{\wedge ^2 V})$, or equivalently an element of $H^1 (G_K ,V(-1)\otimes \overline{\wedge ^2 V})$. The following theorem of Hain and Matsumoto characterises this extension class in terms of the Gross--Kudla--Schoen cycle.
\begin{Theorem}[Hain--Matsumoto \cite{HM}, Theorem 3]\label{hain-matsumoto}
Let $\alpha :\wedge ^3 V\to V\otimes \overline{\wedge ^2 V}$ be the injective homomorphism 
\[
v_1 \wedge v_2 \wedge v_3 \mapsto v_1 \otimes (\overline{v_2 \wedge v_3 })+v_2 \otimes (\overline{v_3 \wedge v_1 })+v_3 \otimes (\overline{v_1 \wedge v_2 }).
\]
Then $[L_2 ]\in H^1 (G_K ,V(-1)\otimes \overline{\wedge ^2 V})$ is equal to $\alpha (-1)_* (\AJ _{\et }[C_b ])$, where $[C_b]$ is the class of the Ceresa cycle in $\CH ^{g-1}(J)$, and $\AJ _{\et }([C_b ])$ is its image in $H^1 (G_K ,\wedge ^3 V(-1))$.
\end{Theorem}
Via the relation between the Ceresa cycle and the Gross--Kudla--Schoen cycle, this has the following corollary.
\begin{Corollary}
The extension class $[L_2 ]\in H^1 (G_K ,V(-1)\otimes \overline{\wedge ^2 V})$ is equal to the image of $\AJ _{\et }([\Delta _{GKS}])$ under the map
\[
H^1 (G_K ,V^{\otimes 3})\to H^1 (G_K ,V\otimes \overline{\wedge ^2 V})
\]
induced by the quotient
\begin{align*}
\tau :V^{\otimes 3} & \to V\otimes \overline{\wedge ^2 V} \\
v_1 \otimes v_2 \otimes v_3 & \mapsto v_1 \otimes \overline{v_2 \wedge v_3 }.
\end{align*}
\end{Corollary}
\begin{proof}
Let $\iota :\wedge ^3 V\to V^{\otimes 3}$ be the inclusion \eqref{iota}, and $\tau ':V^{\otimes 3}\to \wedge ^3 V$ the quotient map $v_1 \otimes v_2 \otimes v_3 \mapsto v_1 \wedge v_3 \wedge v_3 $. By Proposition  \ref{cvg}, the image of $\AJ _{\et}([\Delta _{GKS}])$ in $H^1 (G_K ,\wedge ^3 V(-1))$ under $\tau ' _*$ is equal to $\frac{1}{3}\AJ _{\et }([C_b ])$. Since $\AJ _{\et }([\Delta _{GKS}])$ lies in the image of $\iota _*$, and 
\[
\alpha =3\tau \circ \iota ,
\]
we have
\[
\alpha _* \circ \tau _* ' [\AJ _{\et }([\Delta _{GKS}])]=3\tau _* [\AJ _{\acute{e}t}([\Delta _{GKS}])] \in H^1 (G_K ,V(-1)\otimes \overline{\wedge ^2 V}).
\]
Hence we deduce from Theorem \ref{hain-matsumoto} that
\[
[L_2 ]=\frac{1}{3}\alpha _* \circ \tau _* '[\Delta _{GKS}]=\tau _* [\Delta _{GKS}].
\]
\end{proof}
We now return to the case where $K=\Q $.
Via the commutative diagram
\[
\begin{tikzpicture}
\matrix (m) [matrix of math nodes, row sep=3em,
column sep=3em, text height=1.5ex, text depth=0.25ex]
{\NS (J_{\Q }) & \NS (X_{\Q }) \\
H^2 _{\et } (J_{\overline{\Q }},\Q _p (1)) & H^2 _{\et } (X_{\overline{\Q }},\Q _p (1)), \\};
\path[->]
(m-1-1) edge[auto] node[auto]{$c$} (m-2-1)
edge[auto] node[auto] {$\AJ ^* $ } (m-1-2)
(m-2-1) edge[auto] node[auto] {$\AJ ^* $ } (m-2-2)
(m-1-2) edge[auto] node[auto] {$c$} (m-2-2);
\end{tikzpicture}
\]
(where $c$ denotes the Chern class),
we hence obtain a homomorphism
\begin{align*}
\Ker (\NS (J_{\Q }) & \to \NS (X_{\Q })) \to \ext ^1 (V,\Q _p (1)). \\
[\mathcal{L}] & \mapsto [c([\mathcal{L}])^* ([L_2 ])],
\end{align*}
where $L_2 :=\lie (U_2 )$.
The extensions obtained come from points on $J$. They can be related to the Gross--Kudla--Schoen cycle via the theorem of Hain and Matusmoto (the argument given below follows Darmon, Rotger and Sols \cite{DRS}, who prove a Hodge theoretic analogue of the Lemma below using, using the theorems of Harris and Pulte, which are Hodge theoretic analogues of the Hain--Matsumoto theorem).
\begin{lemma}
Let $Z\subset X\times X$ be a codimension 1 cycle. Let $i_1 ,i_2 ,i_3 :X\hookrightarrow X\times X$ be the closed immersions defined by the subschemes $\{ b\} \times X,X\times \{ b\} $ and the diagonal $\Delta _X$ of $X\times X$ respectively. For $j=1,2,\{1,2\}$, let $i_j ^*$ denote the pull-back morphism
\[
\CH^1 (X\times X)\to \CH^1 (X).
\] 
Then the extension class in $H^1 (G_K ,V)$ associated to the Lie algebra $L_Z$ is given by $\AJ _{\et }(D_Z (b))$, with $D_Z(b)$ as in \eqref{eqDZb}.
\end{lemma}
\begin{proof}
The class $[L_Z]$ is the image of $[L_2 ]$ under the morphism
\[
\ext ^1 _{G_K }(V,\overline{\wedge ^2 V})\to \ext ^1 _{G_{\Q }}(V,\Q _p (1))
\]
induced by $\pi _Z :\overline{\wedge ^2 V}\to \Q _p (1)$.
We have a commutative diagram
\[
\begin{tikzpicture}
\matrix (m) [matrix of math nodes, row sep=3em,
column sep=3em, text height=1.5ex, text depth=0.25ex]
{\CH ^2 (X^3 )_0 & H^1 (G_{\Q },V^{\otimes 3}(-1)) \\
\Pic ^0 (X) & H^1 (G_{\Q },V) \\ };
\path[->]
(m-1-1) edge[auto] node[auto]{$\AJ _{\et }$} (m-1-2)
edge[auto] node[auto] {$\Pi _Z$ } (m-2-1)
(m-1-2) edge[auto] node[auto]{$\Pi _{Z*}$} (m-2-2)
(m-2-1) edge[auto] node[auto]{$\AJ _{\et }$} (m-2-2);
\end{tikzpicture}
\]
By Theorem \ref{hain-matsumoto}, the extension class $[L_2 ]$ is given by $\AJ _{\et }(\Delta _{GKS})$, hence 
\[
[L_Z ]=\Pi _{Z*} ([L_2 ])=\AJ _{\et }(D_Z (b)),
\]
by Lemma \ref{drs-z}.
\end{proof}

\section{Appendix: Proof of the Kolyvagin-Logachev type result}
\label{AppendixKolyvaginLogachev}

	In this appendix, we fix the following notation:
	
	$\bullet$ $M$ is a fixed odd level (which for our applications will be $N$ or $N^2$)
	
	$\bullet$ $f \in S_2(\Gamma_0(M))^{+,\rm{new}}$ is a normalised eigenform.
	
	$\bullet$ $A=A_f$ is its associated quotient of $J_0(M)$, together with the canonical projection $\pi : J_0(M) \ra A$ (independent of the choice of $f$ in its Galois orbit). 
	
	We explain here the following result, attributed to Kolyvagin and Logachev.
	
	\begin{Proposition}[Rank 1 BSD for modular abelian varieties]
		\hspace*{\fill}
		\label{propGZK}
		
		If $L'(f,1) \neq 0$, the rank of $A(\Q)$ is exactly $g:=\dim A$. 
	\end{Proposition}

	\begin{Corollary}
		\label{corKL}
		If $L'(f,1) \neq 0$ for at least two distinct newforms $f$, for the Heegner quotient $A$ of $J_0(M)^{+,\textrm{new}}$ (Definition \ref{defiHeegnerquotient}),
		\[
		\rk(A) = \dim(A) = \rho(A) \geq 2.
		\]
	\end{Corollary}

	\begin{proof}[Proof of the Corollary]
		By Proposition \ref{propGZK} the rank of $A$ is equal to its dimension as it is true for each of its factors $A_f$. Now, we recall that all endomorphisms of an $A_f$ are symmetric and the latter is of $\GL_2$-type, in particular $\End^\dagger(A_f)$ is of rank $\dim A_f$ (see \S \ref{subsec_Samir_result}) . Finally, for $f,g$ non Galois conjugates, there is no morphism between $A_f$ and $A_g$ (by multiplicity one in the newpart) so the endomorphism ring splits and we get the last equality.
	\end{proof}
\begin{Remark}
	 This result is well-known if $\dim A=1$ (\cite{Kolyvagin90} for the original reference, \cite{Gross89} for a survey), and proven in much greater generality in \cite{Nekovar07}, all these along the lines of a stronger result in the rank zero case proved in \cite{KolyvaginLogachev}. It is also (a slightly weaker version of) the main result in Tian's thesis \cite{tian} and of a paper of Tian and Zhang in preparation \cite{TianZhangprep} for which we could not find quotable material. In any case, we felt it sufficiently different from the former references (to which we borrow constantly) to deserve a proof for the nonexperts. For the same reasons, we will simply refer to those papers for parts of the proofs which generalise seamlessly and focus on the more technical points.
\end{Remark}
	\textbf{Convention}
	We use a well-chosen prime number $p$ to obtain Proposition \ref{propGZK}. As we only need one such $p$, in all this appendix, when a property holds when $p$ is large enough, we then automatically assume it is without further mention.
	
		We will prove Proposition \ref{propGZK} by reducing it successively to other statements which will be emphasized.

	\textbf{Notation}
	
	Throughout this text, $\tau$ denotes the usual complex conjugation and when it acts on an $\Z$-module $\Mcal$,  $\Mcal^+$ and $\Mcal^-$ denote the spaces of $m \in \Mcal$ respectively fixed and reversed by $\tau$.	 If $\Mcal$ is finite of odd order, $\Mcal = \Mcal^+ \oplus \Mcal^-$, which we will frequently use implicitly.
	
	Given an Galois extension $L/K$ of number fields and $\gP$ a prime ideal of $L$ unramified over $\gp$, $(\gP,L/K)$ denotes the Frobenius of $\gP$ for this extension, and $(\gp,L/K)$ the conjugacy class of such Frobenius's in $\Gal(L/K)$.

	\subsection{\texorpdfstring{Structure of the $p$-torsion and reduction to Selmer groups}{Structure of the p-torsion and reduction to Selmer groups}}
	
	Let $K_f$ be the number field of coefficients of $f$. By (\cite{KolyvaginLogachev}, section 2.1), there is an isomorphism $[\cdot]: \, K_f \ra \End_\Q A \otimes \Q$ such that for every prime $\ell \nmid N$,  $[a_\ell(f)] \in \End_\Q A$ and 
	\begin{equation}
	\label{eqmultHecke}
	[a_\ell(f)] \circ \pi = \pi \circ T_\ell.
	\end{equation}
	The inverse image of $\End_\Q A$ is thus an order in $K_f$ denoted by $\Ocal$, and $A$ is endowed with a structure of $\Ocal$-module.
	
	We now fix $p$ an odd prime totally split in $K_f$ and prime to the conductor of $\Ocal$ (there are infinitely many such primes by Cebotarev density theorem), so that  $p \Ocal= \gP_1 \cdots \gP_g$ as a decomposition into prime ideals. In all the following, the notation $\gP$ will run through $\gP_1, \cdots, \gP_g$. 
	
	\begin{Remark}
		It is likely the proof still holds for any type of decomposition of $p$ but this hypothesis makes the exposition much more symmetric (and there are infinitely many of them so we can choose it as large as necessary). In the opposite situation, if there is an inert prime in $K_f$, the proof should be a bit simpler.
	\end{Remark}
	
	One of the key ideas to get closer to the case of elliptic curves is decomposing every structure of $\Ocal/(p)$-modules using those prime ideals. Our tool is the following Lemma, often used without mention.
	
	\begin{lemma}
		\label{lemOcalp}
		By the Chinese remainder theorem, $	\Ocal/(p) \cong \bigoplus_{\gP} \Ocal/\gP$,
		in particular each $\Ocal/\gP$ is projective and flat over $\Ocal/(p)$. 
		Every $\Ocal/(p)$-module $\Mcal$ splits canonically into sub-$\Ocal/(p)$-modules
		\[
		\Mcal = \bigoplus_{\gP} \Mcal[\gP], \quad \Mcal[\gP] = \{ m \in \Mcal, \, \gP \cdot m = 0 \} \cong \Mcal/ \gP M,
		\]
		 and projections are given by elements of $\Ocal$. All these isomorphisms are canonical, and for every $m \in \Mcal$, we will denote by $m_\gP$ its projection onto $\Mcal[\gP]$ (or in $ \Mcal /\gP \Mcal$ depending on the context).
	\end{lemma}
	
	\begin{proof}
		The $\gP$ are pairwise coprime so the Chinese remainder theorems holds, and tensoring $\Mcal$ by $\Ocal/(p)$ on one hand fixes it and the other one decomposes it canonically into $\bigoplus_\gP \Mcal/\gP \Mcal$. The latter clearly identifies each $\Mcal/\gP \Mcal$ with the $\gP$-torsion part of $\Mcal$, and the other statements follow.
	\end{proof}
	
	The $\Ocal$-linear representation $A[p]$ of $\GalQ$ thus splits into $\bigoplus_{\gP} A[\gP]$	and for any extension $L$ of $\Q$, we have canonical isomorphisms of $\Ocal/(p)$-modules
	\begin{equation}
	\label{eqcanident}
	(A(L)/pA(L)) [\gP] \cong A(L)/\gP A(L) \quad H^1(L,A[p])[\gP] \cong  H^1(L,A[\gP]).
	\end{equation}
	If $L$ is a number field, for every place $v$ of $L$, the natural localisation maps $\loc_v$ give rise to a commutative diagram
	\begin{equation}
	\label{eqdiagH1An}
	\begin{tikzcd}
	0 \arrow[r] & A(L)/\gP A(L) \arrow[r,"\delta"] \arrow[d, "\loc_v"] & H^1(L,A[\gP]) \arrow[r] \arrow[d, "\loc_v"] & H^1(L,A)[\gP] \arrow[r] \arrow[d, "\loc_v"] & 0 \\
	0 \arrow[r] & A(L_v)/\gP A(L_v) \arrow[r, "\delta_v"]  & H^1(L_v,A[\gP]) \arrow[r] & H^1(L_v,A)[\gP] \arrow[r] & 0,
	\end{tikzcd}
	\end{equation}
	inherited by flatness from the commonly known analogous diagram for the ideal $(p)$ (for references on those facts and the Selmer groups, see \cite{HindrySilverman}, Appendix C.4). Let us define the $\gP$-Selmer group as 
	\begin{equation}
	\label{eqdefgPSel}
	\Sel_\gP(L,A) := \{ s \in H^1(L,A[\gP]), \forall v, \loc_v s \in \delta_v(A(K_v)/\gP A(K_v)) \},
	\end{equation}
	again canonically identified to $\Sel_{p}(L,A)[\gP]$ hence fitting by the same arguments into the exact sequence
	\begin{equation}
	\label{eqsuitexSelTSgP}
	\begin{tikzcd}
	0 \arrow[r] & A(L)/\gP A(L) \arrow[r, "\delta"] & \Sel_\gP(L,A) \arrow[r] & \Sha(L,A)[\gP] \arrow[r] & 0.
	\end{tikzcd}
	\end{equation}

	Now, consider an imaginary quadratic field $K$ whose discriminant $D_K <-4$ is squarefree, prime to the level $M$ and a square modulo $M$. These conditions guarantee that there is a Heegner point (we fix definitively $\gn$ and $[\ag_0]$)
	\begin{equation}
	\label{eqdefHeegnerx}
	x = \left( \Ocal_K, \ng, [\ag_0] \right) \in X_0(M)(H)
	\end{equation}
	in the notation of \cite{Gross84}, where $H$ is the Hilbert class field of $K$.
	As $f_{|w_M} = f$, $\pi \circ w_M = \pi$ therefore by elementary properties of Heegner points (\cite{Gross84}, formulas (4.1) to (5.2)), for $y_1 = \pi ((x)-(\infty)) \in A(H)$, one has
	\begin{equation}
	y_K := \Tr_{H/K} y_1 = \pi \left( \sum_{[\ag] \in \Cl(K)} (\Ocal_K,\ng,[\ag]) - h_K (\infty) \right) \in A(K), 
	\end{equation}
	\begin{equation}
	\label{eqyKalmostdefoverQ}
	\tau(y_K) = \pi \left( \sum_{[\ag] \in \Cl(K)} w_M \cdot (\Ocal_K,\ng,[\ag]) - h_K (\infty) \right) \in y_K + A(\Q)_{\rm{tors}},
	\end{equation}
	
	Now, using a theorem of Waldspurger \cite[Théorème 2.3]{Vigneras81}, let us fix once and for all a $K$ such that $L(f \otimes \varepsilon_K,1) \neq 0$ where $\varepsilon_K$ is the Dirichlet character associated to $K$. By Gross--Zagier formula (\cite{GrossZagier86}, Theorem I.6.3), the point $y_K$ is then nontorsion in $A(K)$ and has an integer multiple  in $A(\Q)$ by \eqref{eqyKalmostdefoverQ}. The subgroup $\Ocal\cdot y_K$ is thus a subgroup of $A(K)$ of rank $g$ (as nonzero elements of $\Ocal$ act by isogenies), which leads us to 
	
	\textbf{Reduction 1} \enquote{Prove that $\Ocal \cdot y_K$ is of finite index in $A(K)$}.
	
	Now, for $p$ large enough, 
	\begin{equation}
	\label{eqhypy1}
	y_K \notin \gP A(K) \textrm{  for all  }\gP,
	\end{equation}
	which further leads by \eqref{eqsuitexSelTSgP} to
	
	\textbf{Reduction 2} \enquote{Prove that for all $\gP$, $\delta(\overline{y_K})$ generates $\Sel_\gP(K,A)$}.
	
	\begin{proof}
	If this claim holds, every $\Sel_\gP(K,A)$ is an $\Ocal/\gP \cong \F_p$-vector space of dimension 1, so $A(K)/\gP A(K)$ is of dimension at most 1 by \eqref{eqsuitexSelTSgP}, and 
	\[
	A(K)/pA(K) \cong \bigoplus_\gP A(K)/\gP A(K)
	\] is of dimension at most $g$ over $\F_p$. This imposes that the Mordell--Weil rank of $A(K)$ over $\Z$ is at most $g$, hence the equality using $\Ocal \cdot y_K$.
	\end{proof}
	
	To conclude this paragraph, $\tau$ acts naturally on $A(\Qb), A[\gP]$, $H^1(K,A[\gP])$ and $\Sel_\gP(K,A)$, and the action of $\Ocal$ and the morphisms between those in \eqref{eqcanident} and \eqref{eqdiagH1An} are $\tau$-equivariant. We fix from now on a polarisation $A \rightarrow \widehat{A}$ of degree prime to $p$ (otherwise choose a larger prime $p$), which thus defines a Weil pairing $A[p] \times A[p] \ra \mu_p$. Its elementary properties (\cite{MilneAbVar86}, Lemma 16.2) then imply the following structural result, crucial for our understanding.
	
	\begin{lemma}
		\label{lemWeil}
		For every $\gP$ and $\varepsilon = \pm 1$: 
		
		$\bullet$ The $2g$ spaces $A[\gP]^\varepsilon$ are pairwise orthogonal for the Weil pairing, except the $A[\gP]^\varepsilon$ with the same $\gP$ and opposite sign.
		
		$\bullet$ The two spaces $A[p]^\varepsilon$ are isotropic for the Weil pairing
		
		$\bullet$ Each $A[\gP]^\varepsilon$ is then of dimension 1 over $\F_p$ and $\dim_{\F_p} A[\gP] = 2$.
	\end{lemma} 

	\subsection{Pairing the Galois group and Selmer groups, and Kolyvagin primes}

	Throughout this appendix, we fix 
	\[
	L:=K(A[\gP]), \quad G:=\Gal(L/K).
	\]
	(notice $L$ is Galois over $\Q$ ).
	\begin{Proposition}
		\label{propresmorph}
		For $p$ large enough:
		
		$(a)$ $A[\gP]$ is (absolutely) irreducible as a representation of $\GalQ$.
		
		$(b)$ The canonical restriction morphism 
		\[
		H^1(K,A[\gP]) \overset{\operatorname{res}}{\ra} H^1(L,A[\gP])^G = \Hom_G(\Gal(L^{\rm{ab}}/L),A[\gP])
		\]
	 is injective, with the action of $G$ on $\Gal(L^{\rm{ab}}/L)$ defined by conjugation in $\GalQ$.
	\end{Proposition}

\begin{Remark}
	Here is an important difference with the $\dim A=1$ case: the Galois representation $\GalQ \ra \GL(A[\gP]) \cong \GL_2(\F_p)$ is not proven to be surjective (\cite{Ribet76} does not cover the square $M$ case), but we will manage with $(a)$ and $(b)$ although it introduces significant changes compared to some arguments in \cite{Gross89}.
\end{Remark}
\begin{proof}
	$(a)$ is Lemma 3.2 of \cite{Ribet04} and $(b)$ is Proposition 6.1.2 of \cite{Nekovar07}.
\end{proof}

	We now choose $S$ a finite sub-$\Ocal$-module of $H^1(K,A[\gP])$, stable by $\tau$ (this will be $\Sel_\gP(K,A)$ and then an auxiliary module for the proof). By Proposition \ref{propresmorph} $(b)$, there is a pairing 
	\[
	\fonctionsansnom{S \times \Gal(L^{\rm{ab}}/L)}{A[\gP]}{(s,\sigma)}{\operatorname{res}(s)(\sigma)}
	\] 
	which is injective on the left. We define $L_S$ the extension of $L$ whose absolute Galois group is the orthogonal of $S$, and thus obtain a nondegenerate pairing between finite abelian $p$-torsion groups
	\[
	[\cdot,\cdot]_S : S \times H_S \rightarrow A[\gP], \quad H_S := \Gal(L_S/L).
	\]
 Keeping track of the actions of $\tau$ and the $\sigma \in G$, we have that
	\begin{equation}
	\label{eqGaloisconjgpairingS}
	\tau[s,\rho]_S = [\tau(s),\tau \rho \tau^{-1}]_S, \quad \sigma[s,\rho] = [s,\sigma \rho \sigma^{-1}].
	\end{equation}
	In particular, the extension $L_S/\Q$ is Galois. 
	
	\begin{lemma}
		\label{lemperfdual}
		This pairing induces a perfect bilinear pairing from $S^\varepsilon \times H_S^+$ to $A[\gP]^\varepsilon \cong \F_p$, hence a duality between $S^\varepsilon$ and $H_S^+$.
	\end{lemma}

	\begin{proof}
		By \eqref{eqGaloisconjgpairingS}, these two pairings (for $\varepsilon= \pm 1$) are well-defined, let us prove they are injective on the left and on the right, they will then be perfect as everything is finite(-dimensional). For $s \in S^\varepsilon$, if $[s, H_S^+]_S=0$,then 
		\[
		[s,H_S]_S = [s,H_S^-]_S \subset A[\gP]^{-\varepsilon}
		\] by the same arguments, but $[s,H_S]_S$ is stable by $\GalQ$ by \eqref{eqGaloisconjgpairingS} again. As $A[\gP]$ is irreducible by Proposition \ref{propresmorph} $(a)$, it imposes $[s,H_S]_S=0$ therefore $s=0$ by nondegeneracy. Now, assume $[S^\varepsilon,h]_S=0$ for some $h \in H_S^+$. This holds for all conjugates $\sigma h \sigma^{-1}$ of $h$ in $H_S$ by \eqref{eqGaloisconjgpairingS}, so on the group $H' \subset H_S$ they generate. Again, this forces $[S,H']_S \subset A[\gP]^{-\varepsilon}$, but this group is stable by $\GalQ$ hence $H'=0$.
	\end{proof}

	\begin{lemma}
		\label{lemortho}
		Fix $\varepsilon=\pm 1$ and $I_S^+$ a proper subgroup of $H_S^+$. Then, $s \in S^\varepsilon$ is 0 if for all $\rho \in H_S^+ \backslash I_S^+$, $[s,\rho]_S =0$.
	\end{lemma}

\begin{proof}
	It is a trivial consequence of the perfect duality above, knowing that the sub-$\F_p$-vector space generated by $ H_0^+ \backslash I_0^+$ is $H_0^+$ itself, e. g. by a counting argument.
\end{proof}

	\textbf{Reduction 3} \enquote{For all $\gP$, apply Lemma \ref{lemortho} to ($s_0=0$, $\varepsilon=-1$) (resp.  $\delta \overline{y_K}$, $\varepsilon=1$) to prove that $\Sel_\gP(K,A)^- = 0$ (resp. $\Sel_\gP(K,A)^+ = \langle  \delta \overline{y_K} \rangle$)}.
	
	The next subsection will show us how to compute the pairing $[\cdot,\cdot]_S$.
	
	\subsection{Kolyvagin primes }

	\begin{Definition}
		\hspace*{\fill}
		\begin{itemize}
		\item A \textit{Kolyvagin prime} $\ell$ is a prime number such that:
		
		$-$ $\ell$ does not divide $D_K M p$ (or the conductor of $\Ocal$), so is unramified in $L$.
		
		$-$ The conjugacy class of $(\ell,L/\Q)$ is the one of $\tau$ in $\Gal(L/\Q)$. In particular, $\ell \Ocal_K =: \lambda_\ell$ is inert over $\ell$. We will often shorten it to $\lambda$ if $\ell$ is nonambiguous, and for any extension $K'$ of $K$, $\lambda_{K'}$ will be a choice of prime ideal of $\Ocal_{K'}$ above $\lambda$ (in a consistent fashion if multiple extensions are considered). 
		
		\item A \textit{Kolyvagin number} $n$ is a squarefree product of Kolyvagin primes $\ell$.
	\end{itemize}
	\end{Definition}

	In the same fashion as in (\cite{Gross84}, (3.3)), Kolyvagin primes have many strong properties.
	
	\begin{Proposition}
		\label{propKolprime}
		For a Kolyvagin prime $\ell$, $\lambda$ splits completely in $L$. Furthermore: 
		\[
		p | a_\ell(f), \quad p | \ell+1 
		\]
		in $\Ocal$, and all the points of $A[\gP]$ are defined over $K_{\lambda}$. Moreover, the two eigenspaces $(A(K_\lambda)/ \gP A(K_\lambda))^\pm$ for the  action of $\Frob(\ell)$ are of dimension 1 over $\F_p$.
	\end{Proposition}

	\begin{proof}
		Up to conjugation, $(\lambda_L,L/K) = (\lambda_L,L/\Q)^{f(\lambda/\ell)} = \tau^2 = \Id$ so $\lambda_L/\lambda$ is totally split. Now, by Eichler-Shimura theory (\cite{KolyvaginLogachev}, formula (2.1.8)), the characteristic polynomial of the Frobenius endomorphism $\Frob(\ell)$ on the reduction $\widetilde{A}$ of $A$ modulo $\ell$ (as an $\Ocal$-linear endomorphism) is $X^2 - a_\ell(f) X + \ell$ and the one of the complex conjugation is $X^2-1$, and they must agree on $\widetilde{A}[p]$. In particular, $\Frob(\ell)^2$ acts trivially on $A[\gP]$ so $\widetilde{A}[\gP] = \widetilde{A}[\gP] (\F_\lambda)$ and we can lift those points to $K_\lambda$. By the same arguments, on also has the decomposition
		\[
		\widetilde{A}[\gP](\F_\lambda) = \widetilde{A}[\gP](\F_\lambda)^{+} \oplus \widetilde{A}[\gP](\F_\lambda)^{-}
		\]
		in two nontrivial spaces, given the characteristic polynomial of $\Frob(\ell)$, so each of the two spaces on the right-hand side is of dimension 1 over $\F_p$. We deduce immediately by the structure of finite abelian groups that as groups,
		\[
		(\widetilde{A}(\F_\lambda)^\varepsilon/\gP \widetilde{A}(\F_\lambda) ^\varepsilon) \cong \widetilde{A}(\F_\lambda)^\varepsilon[\gP],
		\]
		which proves that each $(\widetilde{A}(\F_\lambda)/\gP \widetilde{A}(\F_\lambda))^{\varepsilon}$ must be of dimension 1 over $\F_p$, and this also lifts to $K_\lambda$ (without increasing the dimension as the group of elements reducing to 0 modulo $\lambda$ is $p$-divisible).
	\end{proof}

	To state the next result, recall that for a finite place $v \nmid p$ 
of good reduction of $A$, the image of $A(K_v)/pA(K_v)$ in $H^1(K_v,A[p])$ is precisely the inflation of $H^1(K_v^{\textrm{unr}}/K_v,A[p])$, called the unramified part. The latter is isomorphic to $A[p]$ when all the $p$-torsion is defined over $K_v$, via the evaluation of the cocycles at $\Frob(v)$ the topological generator of $\Gal(K_v^{\textrm{unr}}/K_v)$. The same argument translates for $A[\gP]$ by tensoring by $\Ocal/\gP$ again.
	
	\begin{Proposition}
		\label{propLcaleval}
		Let $\Lcal$ be an unramified prime ideal of $L_S$ whose Frobenius in $\Gal(L_S/\Q)$ is $\tau h$ for $h \in H_S$. It is above a Kolyvagin prime $\ell$ and for every $s \in S$ whose localisation at $\lambda = \ell \Ocal_K$ is unramified,
		\[
		[s,(\tau h)^2]_S = \operatorname{ev}_\lambda(s) := (\loc_\lambda s)(\Frob(\lambda)) \in A[\gP].
		\]
		through the identification  described above, as all $A[\gP]$ is defined over $K_\lambda$.
	\end{Proposition}
	
	\begin{proof}
		By hypothesis, $(\Lcal,L/\Q)_{|L} = \tau$ so $\lambda_L = \Lcal \cap \Ocal$ is indeed above a Kolyvagin prime $\ell$. On the other hand, $(\Lcal,L_S/L)=(\Lcal,L_S/\Q)^2 = (\tau h)^2$ as the inertia does not change between $K$ and $L_S$. 	Now, the diagram
		\[
		\xymatrix{
			S \subset H^1(K,A[\gP]) \ar[r]^{\operatorname{res}} \ar[d]_{\loc_\lambda} & \Hom_G(H_S,A[\gP]) \ar[d]^{\operatorname{ev}_{(\Lcal,L_S/L)}} \\
			H^1_{\textrm{unr}}(K_\lambda,A[\gP]) \ar[r]^-{\operatorname{ev}_{\Frob(\lambda)}} &  A[\gP]
		}
		\]
		is clearly commutative, which establishes the equality by definition.
	\end{proof}

	\begin{Remark}
		The set of all $(\tau h)^2$ thus obtained is exactly $H_S^+$, by Cebotarev density theorem.
	\end{Remark}
	
	Now, for any place $v$ of $K$, we can construct (\cite{Tate58}, section 2) a canonical bilinear pairing obtained from Tate duality 
	\begin{equation}\label{tate}
	\langle \cdot ,\cdot \rangle_{K_v} : A(K_v)/pA(K_v) \times H^1(K_v,A)[p] \rightarrow \Br(K_v)[p] \cong \Z/p\Z.
	\end{equation}
	The key use of Tate duality is the following Proposition, which is a slight generalisation of \cite[Proposition 8.2]{Gross89}.
	\begin{Proposition}
		\label{propTate}
		If for a prime $\lambda$ of $K$ (above a Kolyvagin prime) and a $\gamma \in H^1(K,A)^{\varepsilon}[\gP]$, one has $\loc_v \gamma = 0$ for all $v \neq \lambda$ but $\loc_\lambda \gamma \neq 0$, then for every $s \in \Sel_\gP(K,A)^{\varepsilon}$, $\loc_\lambda s=0$.
	\end{Proposition}
	
	\begin{proof}
		 By its definition, \eqref{tate} comes from the Weil pairing in the sense that the latter induces a cup product 
		\[
		(\cdot,\cdot)_{K_v}: H^1(K_v,A[p]) \times H^1(K_v,A[p]) \ra H^2(K_v,\mu_p) = \Br(K_v)[p],
		\]
		for which $\delta_v(A(K_v)/pA(K_v))$ is isotropic, and the resulting quotiented pairing is exactly $\langle \cdot ,\cdot \rangle_{K_v}$. Now, the so-called  \textit{global Tate duality} states that for any $s \in \Sel_p(K,A)$, $\gamma \in H^1(K,A)[p]$, 
		\[
		\sum_{v \in M_K} \operatorname{inv}_v \langle \delta_v^{-1} \loc_v s, \loc_v \gamma \rangle_{K_v} = 0 \in \Q/\Z,
		\] 
		where $\operatorname{inv}_v : \Br(K_v) \ra \Q/\Z$ is the Brauer invariant isomorphism for all $v$.
		Indeed, let us lift $\gamma$ to $\widetilde{\gamma} \in H^1(K_v,A[p])$, so that for every $v \in M_K$, 
		\[
		\langle \delta_v^{-1} \loc_v s, \loc_v \gamma \rangle_{K_v} = (\loc_v s, \loc_v \widetilde{\gamma})_{K_v} = \loc_{v,\rm{Br}} (s,\widetilde{\gamma})_K
		\]
		with the analogous definition of $(\cdot,\cdot)_{K}$ on $K$, and $\loc_{v,\rm{Br}}: \Br(K) \ra \Br(K_v)$ the usual localisation. Now, by properties of Brauer groups, the sum of $\operatorname{inv}_v \circ \loc_v$ is 0 on $\Br(K)$ hence the formula.
		
		Under our assumptions on $\gamma$ and $s$, we thus have $\loc_\lambda \gamma \neq 0$ and $ \langle \delta_\lambda^{-1} \loc_\lambda s, \loc_\lambda \gamma \rangle_{K_\lambda} = 0$, let us show how this  implies that $\loc_\lambda s =0$.
		 
		 By the original arguments of \cite{Tate58}, the pairing $\langle \cdot, \cdot, \rangle_{K_\lambda}$ is a perfect pairing. Being inherited from the Weil pairing, the $\gP$ and $\gP'$-parts for $\gP \neq \gP'$ are orthogonal, so it induces a duality
		 \[
		 A(K_\lambda) / \gP A(K_\lambda) \times H^1(K_\lambda,A)[\gP] \rightarrow \Z/p\Z.
		 \]
		 Now, it is also invariant by $\Gal(K_\lambda/\Q_\ell)$-action (there is a difference with the Weil pairing here, but it is also inherited from the cup product $(\cdot,\cdot)_{K_\lambda}$), so the $+$ and $-$ spaces on each side are orthogonal. We thus have for $\varepsilon=\pm 1$ a duality
		 \[
		 (A(K_\lambda) / \gP A(K_\lambda))^\varepsilon \times H^1(K_\lambda,A)^\varepsilon[\gP] \rightarrow \Z/p\Z,
		 \]
		 but making use of the fact that $\lambda$ is above a Kolyvagin prime, each space of the duality is thus of dimension 1 over $\F_p$ (Proposition \ref{propKolprime}), and so the pairing can be 0 only if one of the terms is 0, hence $\loc_\lambda s=0$.
	\end{proof}

	\subsection{Construction of the Kolyvagin classes}
	
	Following \cite{KolyvaginLogachev}, one takes the classes $[\ag]$ and prime ideal $\ng$ induced by the choices made in \eqref{eqdefHeegnerx} on orders of $\Ocal_K$, and for any Kolyvagin number $n$, we get Heegner points
	\[
	x_n = (\Z + n \Ocal_K,\ng \cap (\Z + n \Ocal_K),[\ag]), \quad y_n = \pi ((x_n)- (\infty)) \in A(K_n),
	\]
	where by class field theory, $K_n$ is the class ring field of conductor $n$ ($K_1=H$).
	
	The  notation $\lambda_{n,\ell}$ will refer to a choice of prime ideal of $K_n$ above $\ell$ a Kolyvagin prime, consistent in case of towers of extensions, shortened to $\lambda_n$ if there is no doubt on $\ell$.  One has that $G_n := \Gal(K_n/K_1) \cong (\Ocal_K/n \Ocal_K)^* / (\Z/n\Z)^*$ and the following diagrams for $n=\ell m$ by class field theory:
	
	\begin{equation}
	\label{eqdiagsCFT}
	\xymatrix{
& K_n \ar@{-}[ld]^{G_\ell} \ar@{-}[rd]_{G_m} \ar@{-}[dd]^{G_n} & && & \lambda_n \ar@{-}[ld]_{\textrm{tot.ram.}} \ar@{-}[rd]^{\textrm{tot.sp.}} & \\
	K_m \ar@{-}[rd]^{G_m} &  & K_\ell \ar@{-}[ld]^{G_\ell} && \lambda_m \ar@{-}[rd]_{\textrm{tot.sp.}} &  & \lambda_\ell \ar@{-}[ld]^{\textrm{tot.ram.}} \\
	& K_1=H \ar@{-}[d]^{\Cl(K)} & && & \lambda_1 \ar@{-}[d]^{\textrm{tot.sp.}} & \\
	& K \ar@{-}[d]^{\langle 1,\tau \rangle} & && & \lambda \ar@{-}[d]^{\textrm{in.}} & \\
		& \Q & && & \ell &
}
	\end{equation}
In particular, $\F_{\lambda_n} = \F_{\lambda_m}= \F_{\lambda}$, a fact which will be ubiquitous and used without further mention in the end of the argument.
 
The crucial properties of these points (making them a \enquote{Kolyvagin system}) are the following, $\widetilde{A}$ denoting the (good) reduction of $A$ modulo $\ell$ and $\Frob(\ell)$ the associated Frobenius endomorphism on $\widetilde{A}$. 
\begin{Proposition}
	For $n = \ell m$ a Kolyvagin number,
\begin{eqnarray}
\label{eqTryn}
\Tr_{K_n/K_m} y_n & = & [a_\ell(f)] y_m \in A(K_m) \\
\label{eqynmodlambda}
 y_n \mod \lambda_{n} & = & \Frob(\ell) \cdot y_m \, \, {\rm{  in  }} \, \,   \widetilde{A}(\F_{\lambda_n}) = \widetilde{A}(\F_{\lambda}) \\
 \label{eqynconj}
\tau(y_n) & \in & \sigma(y_n) + A(K_n)_{\rm{tors}}
\end{eqnarray}
for some $\sigma \in \Gcal_n := \Gal(K_n/K)$.
\end{Proposition}

\begin{proof}
	
	By classical properties of Heegner points (\cite{Gross84}, paragraphs 4 and 5) and class field theory for $K_n/K_m$,
	\begin{equation}
	\label{eqTrxn}
	\Tr_{K_n/K_m} x_n = T_\ell \cdot x_m
	\end{equation} as divisors on $X_0(N)$, which proves \eqref{eqTryn} when combined with \eqref{eqmultHecke}. We obtain \eqref{eqynconj} with the same properties.
	
	Looking at the diagrams \eqref{eqdiagsCFT}, as $\lambda_n/\lambda_m$ is totally ramified, the reduction of the left-hand side of \eqref{eqTrxn} is $(\ell+1) x_n \mod \lambda_n$, and the one of the right-hand side has one term equal to $\Frob(\ell) x_m$ by the Eichler-Shimura relation $T_\ell = \Frob(\ell) + \widehat{\Frob(\ell)}$, so there exists $\sigma \in \Gal(K_n/K_m)$ such that the reduction of $\sigma x_n$ is $\Frob(\ell) \widetilde{x_m}$, but every $\sigma$ reduces to the identity on $\widetilde{A}(\F_\lambda)$ so the equality is true term by term hence \eqref{eqynmodlambda}. See also \cite[Corollaries 2.3.3 and 2.3.4]{KolyvaginLogachev} for the $n=\ell$ case.
\end{proof}
	\begin{Proposition}
		For every Kolyvagin number $n$, one can define in successive order (using the Heegner points $y_m$ for $m|n$):
		
		$\bullet$ A point $P_n \in A(K_n)$ whose class $[P_n] \in A(K_n)/p A(K_n)$ is fixed by $\Gcal_n$ (and $P_1=y_K$).
		
	$\bullet$ The unique class $c(n) \in H^1(K,A[p])$ whose restriction to $H^1(K_n,A[p])^{\Gcal_n}$ comes from $[P_n]$, and its image $d(n)$ in $H^1(K,A)[p]$. They correspond to one another in the following commutative diagram with exact rows and columns
	\begin{equation}
	\label{diagKolclass}
	\xymatrix{
		& &  & \overset{\widetilde{d(n)}}{H^1(K_n/K,A)[p]} \ar[d]^-{\operatorname{inf}} & \\
		0 \ar[r] & A(K)/pA(K) \ar[r]^{\delta} \ar[d] & \overset{c(n)}{H^1(K,A[p])} \ar[d]^{\operatorname{Res}}_{\sim} \ar[r] & \overset{d(n)}{H^1(K,A)[p]} \ar[r] \ar[d]^{\operatorname{Res}} & 0\\ 
		0 \ar[r]  & \overset{[P_n]}{(A(K_n)/pA(K_n))^{\Gcal_n}} \ar[r]^{\delta_n} & H^1(K_n,A[p])^{\Gcal_n} \ar[r] & H^1(K_n,A)^{\Gcal_n}[p]
	}
	\end{equation}
	\end{Proposition}
	
	\begin{proof}
		The construction and properties of $P_n$ proceeds exactly as in (\cite{Gross89}, (3.5) to (4.1)). The only nontrivial thing to prove (to define $c(n)$ from $[P_n]$)is that the central row of \eqref{diagKolclass} is an isomorphism. The extension $K_n/\Q$ is unramified outside primes dividing $D_K n$, and the extension $\Q(A[p])/\Q$ is unramified outside primes dividing  $Mp$, so as $D_Kn$ and $pM$ are coprime by construction, these extensions are linearly disjoint. In particular, $K_n(A[p])/K_n$ has Galois group isomorphic to $\Gal(\Q(A[p])/\Q)$ and thus no fixed point in $A[p]$ by Proposition \ref{propresmorph} $(a)$. The isomorphism follows by (\cite{Gross89}, (4.2)). 
	\end{proof}
	These points enjoy a wealth of very strong properties detailed below.
	
	\begin{Proposition}
		\label{propclassesKol}
		For every Kolyvagin number $n$:
		
		$(a)$
			$[P_n]$ (resp. $c(n),d(n)$) lives in the $\mu(n)$-eigenspace of $A(K_n)/pA(K_n)$ (resp. $H^1(K,A[p])$, $H^1(K,A)[p]$), where $\mu(n)$ is the Moebius function.
			
			$(b)$ The class $c(n)_\gP \in H^1(K,A[\gP])$ (resp. $d(n)_\gP  \in H^1(K,A)[\gP]$) is trivial if and only if $P_n \in \gP A(K_n)$ (resp. $\gP A(K_n) + A(K)^{\mu(n)}$). 
			
			$(c)$ For every place $v$ of $K$, the class $\loc_v d(n)$ is trivial except if $v|n$.
			
			$(d)$ If $n=\ell m$ and $\lambda = \ell \Ocal_K$, the class $\loc_{\lambda} d(n)_\gP$ is trivial if and only if $P_m \in \gP A(K_{\lambda_m})$ if and only if $\loc_\lambda c(m)_\gP = 0$.
	\end{Proposition}

\begin{proof}
	$(a)$ for $[P_n]$ is inherited from \eqref{eqynconj} by the construction of $P_n$ (see Proposition 5.4 of \cite{Gross89}), and deduced for $c(n)$, $d(n)$ by $\tau$-equivariance of the morphisms of \eqref{diagKolclass}. 
	
	$(b)$ is obtained by tensoring \eqref{diagKolclass} by $\Ocal/\gP$, which preserves exactness by flatness and $[P_n]$ seen in $A(K_n)/pA(K_n) \otimes \Ocal/\gP$ is exactly the image of $P_n$ in $A(K_n)/\gP A(K_n)$. The proof of $(c)$ is given by Proposition 6.2 of \cite{Gross89}. 
	
	For $(d)$, define $D=\Gal((K_n)_{\lambda_n}/K_\lambda)$, which is cyclic generated by some $\sigma_\ell$. We thus have injective arrows (defined below)
	\begin{equation}
	\label{eqseqH1red}
	H^1(D, A)[p] \overset{\operatorname{red}}{\hookrightarrow} \widetilde{A}(\F_\lambda)[p] \cong H^1(\F_\lambda,\widetilde{A}[p]) \overset{\iota}{\hookleftarrow} \widetilde{A}(\F_\lambda)/p \widetilde{A}(\F_\lambda) 
	\end{equation}
	where for a cocycle $c  \in Z^1(D, A)$,  $\operatorname{red}(c) = c(\sigma_\ell) \mod \lambda_n$, and invariant up to coboundary because $K_n/K_m$ is totally ramified at $\lambda_m$, so $\operatorname{red}$ is well-defined. As $A^1((K_n)_{\lambda_n})$ is a pro-$\ell$-group, $H^1(D,A^1)[p]=0$ which proves that  $\operatorname{red}$ is injective. The map $\iota$ is the quotiented connecting homomorphism, automatically injective. As  $\widetilde{A}(\F_\lambda)$ is a finite abelian group, the orders of $\widetilde{A}(\F_\lambda)[p]$ and $ \widetilde{A}(\F_\lambda)/p \widetilde{A}(\F_\lambda) $ are readily seen to be equal so $\iota$ is also an isomorphism. By (\cite{Gross89}, Proposition 6.2 (2)), the image of $\loc_\lambda d(n)$ in $\widetilde{A}(\F_\lambda)[p]$ by $\operatorname{red}$ is
	\[
	((\ell+1) \Frob(\ell) - [a_\ell(f)]) \cdot \widetilde{R_m}, 
	\]
	where $\widetilde{R_m}$ is any choice of $p$-th root of $\widetilde{P_m}$ in $\widetilde{A}$. By the proof of Proposition \ref{propKolprime}, its image by $\Frob(\ell)$ is then 
	\[
	\ell(\Frob(\ell)^2 - \Id) \widetilde{R_m} = - (\Frob(\ell)^2 - \Id) \widetilde{R_m},
	\]
	but the injection $\iota$ from \eqref{eqseqH1red} is explicitly given by taking a $p$-th root and applying $(\Frob(\ell)^2 - \Id)$, as $\Frob(\ell)^2 = \Frob(\lambda)$ (\cite{KolyvaginLogachev}, Lemma 3.4.2 for details). The image of $\loc_\lambda d(n)$ in $\widetilde{A}(\F_\lambda)/p \widetilde{A}(\F_\lambda)$ via \eqref{eqseqH1red} is thus exactly $-\Frob(\ell)^{-1} \cdot \widetilde{P_m}$, and its $\gP$-part is trivial if and only if the $\gP$-part of $\widetilde{P_m}$ is. Finally,  $A^1(K_{\lambda_m})$ is $p$-divisible hence the equality of $\Ocal/(p)$-modules $A(K_{\lambda_m})/pA(K_{\lambda_m}) \cong \widetilde{A}(\F_\lambda)/p\widetilde{A}(\F_\lambda)$, so finally $\loc_\lambda d(n)_\gP$ is trivial if and only if $[P_m] \in A(K_{\lambda_m})/pA(K_{\lambda_m})[\gP]$, which is equivalent to $P_m \in \gP A(K_{\lambda_m})$ and the equivalence in terms of $c(m)$ is straightforward. 
\end{proof}

\subsection{End of the proof}

Let $S = \Sel_\gP(K,A)$. By \eqref{eqhypy1}, $P_1 = y_K \notin \gP A(K)$, hence it defines a nonzero $s_K:=c(1) \in S^+$ (Proposition \ref{propclassesKol} $(a)$). Fixing $s \in S$, for every $h \in H_S$, by Cebotarev density theorem, there is a prime ideal $\Lcal$ such that $(\Lcal,L_S/\Q) = \tau h$, and by Proposition \ref{propLcaleval},
\[
[s,(\tau h)^2]_S = \loc_\lambda s ( \Frob(\lambda))
\]
where $\lambda$ is the prime ideal of $K$ below $\Lcal$, and above $\ell$ which is a Kolyvagin prime. Outside of $I_S^+$ (defined as the $+$-part of the orthogonal of $s_K$), this formula proves that $\loc_\lambda s_K \neq 0$, so $\loc_\lambda d(\ell)_\gP \neq 0$  and all other localisations of $d(\ell)_\gP$ are trivial by Proposition \ref{propclassesKol}. By Proposition \ref{propTate}, if $s \in S^-$, $ \loc_\lambda s = 0$ so $[s,(\tau h)^2]_S = 0$, hence $S^-=0$ by Lemma \ref{lemortho}. 

Now, consider $s \in S^+$ such that for some $\Lcal$ as above (fixed, so it fixes $\lambda$ and $h$ above), $\loc_\lambda s=0$. We have $\loc_\lambda s_K \neq 0$ by hypothesis on $h$, so in turn $\loc_\lambda d(\ell)_\gP \neq 0$ by Proposition \ref{propclassesKol} $(d)$ and $c(\ell)_\gP$ does not belong to $S$. By the perfect pairing result of Lemma \ref{lemperfdual} applied to $\langle S,c(\ell) \rangle$if $(\tau h)^2 \notin I_S^+$, the extensions $L_S$ and $L_{\langle c(\ell) \rangle}$ are linearly disjoint over $L$, which allows, for any $h' \in H_S$, to choose $\Lcal'$ a prime ideal of $L_S L_{\langle c(\ell) \rangle}$ whose Frobenius restricted to $L_S$ is $\tau h'$ and whose Frobenius restricted to  $L_{\langle c(\ell) \rangle}$ is of the shape $\tau h_0$ and \textit{not} orthogonal to $c(\ell)_\gP$. Denoting $\ell'$ the corresponding Kolyvagin prime and $\lambda'$ the ideal of $\Ocal_K$, we thus have 
\[
[c(\ell)_\gP, (\tau h_0)^2] = \loc_{\lambda'} c(\ell)_\gP (\Frob(\lambda')),
\]
this formula being legitimate because $\loc_{\lambda'}(d(\ell)_\gP) = 0$ by Proposition \ref{propclassesKol} $(c)$. All this proves that $ \loc_{\lambda'} c(\ell)_\gP \neq 0$ so $\loc_{\lambda'} d(\ell \ell')_\gP \neq 0$ by Proposition \ref{propclassesKol} $(d)$, and it belongs to $H^1(K,A)^+[\gP]$. Now, for our $s$ above, the global Tate duality between $s$ and $d(\ell \ell')$ in the proof of Proposition \ref{propTate} has two possible nonzero terms (in $\lambda$ and $\lambda'$ ), but by hypothesis $\loc_\lambda s=0$ so the $\lambda'$-term is alone, therefore 0 as well. This implies by Proposition \ref{propTate} that $\loc_{\lambda'} s = 0$ for all such $\lambda'$, therefore $s=0$ in this case by Lemma \ref{lemortho}.

Finally, for $s \in S^+$, as $\loc_\lambda s_K \neq 0$ and the space $(A(K_\lambda)/\gP A(K_\lambda))^+$ is one-dimensional (Proposition \ref{propKolprime}), there is $k \in \Z$ such that $s - k s_K$ satisfies the previous hypothesis and then $s=k s_K$, so we have proved that $S^+ = \langle s_K \rangle$.

	\bibliographystyle{alphaSLF}
	\bibliography{bibdump}

\end{document}